\documentclass[11pt,a4paper,oneside,reqno]{amsart}

\let\Bbb=\mathbb
\newtheorem{theorem}{Theorem}[section]
\newtheorem{lemma}[theorem]{Lemma}
\usepackage{setspace}
\theoremstyle{definition}
\newtheorem{definition}[theorem]{Definition}
\newtheorem{example}[theorem]{Example}
\newtheorem{proposition}[theorem]{Proposition}
\newtheorem{corollary}[theorem]{Corollary}
\newtheorem{Notation}[theorem]{Notation}
\newtheorem{remark}[theorem]{Remark}
\newtheorem{Notation and Remark}[theorem]{Notation and Remark}
 \long\def\comment#1{\relax}
\long\def\commented_out#1{\relax}
\usepackage{enumitem}
\usepackage{amssymb}
\usepackage{lipsum}
\usepackage{latexsym}
\usepackage{fullpage}
\usepackage{xspace} 
\DeclareMathOperator{\alfa}{\alpha}
\def\natn{\mathbb{N}}
\def\intz{\mathbb{Z}}
\newcommand{\N}[2][k]{N_{\mathbb{Z}_{#2}}(<#1)}
\newcommand{\Nd}[2][k]{N'_{\mathbb{Z}_{#2}}(<#1)}
\newcommand{\Ralfa}{R[\alfa]}

\newcommand{\Null}[1][R]{N_{#1}}
\newcommand{\Nulld}[1][R]{N'_{#1}}
\DeclareMathOperator{\im}{Im}
\newcommand{\Stab}[1][n]{St_{\alfa}(\mathbb{Z}_{#1})}
\newcommand{\PrPol}[1][R]{\mathcal{P}(#1)}
\newcommand{\PolFun}[1][R]{\mathcal{F}(#1)}
\newcommand{\Zm}[1][m]{\mathbb{Z}_{#1}}
\newcommand{\Zmalfa}[1][m]{\mathbb{Z}_{#1} [\alfa]}
\newcommand{\FZm}{\mathcal{F}(\mathbb{Z}_{m})}

\newcommand{\Zpnalfa}{\Zmalfa[p^n]}

\newcommand{\Fq}{\mathbb{F}_q}
\newcommand{\Fqa}{\mathbb{F}_q[\alpha]}
\newcommand{\restrict}[2]{#1\lower2pt\hbox{\big|}_{#2}}
\DeclareMathOperator{\quv}{\xspace{ }\triangleq\xspace{ }}
\DeclareMathOperator{\nquv}{\xspace{}\not\triangleq\xspace{}}

\usepackage{mathtools}

\usepackage{epsf}
\usepackage{layout}
\usepackage{hyperref}

\usepackage{epsf}
\usepackage{layout}
\usepackage{hyperref}
\begin{document}
	
	%\title[Polynomial functions of the ring  of dual numbers modulo $m$]%
	%{Polynomial functions of the ring  of dual numbers modulo $m$}
	\title[Polynomial functions on rings of dual numbers]%
	{Polynomial functions on rings of dual numbers over residue class
		rings of the integers}
	\author[H. Al-Ezeh \and A. A. Al-Maktry\and S. Frisch]%
	{Hasan Al-Ezeh* \and Amr Ali Al-Maktry** \and Sophie Frisch**}
	
	\newcommand{\acr}{\newline\indent}
	
	\address{\llap{*\,}Department of Mathematics\acr
		The University of Jordan\acr
		Amman 11942\acr
		JORDAN}
	\email{alezehh@ju.edu.jo}
	\address{\llap{**\,}Department of Analysis and Number Theory (5010)\acr
		Technische Universit\"at Graz\acr
		Kopernikusgasse~24/II\acr
		8010 Graz \acr
		AUSTRIA}
	\email{almaktry@math.tugraz.at}
	\email{frisch@math.tugraz.at}
	
	%%\acr is not required (if you do not need to see a column);
	%%in our style \\ makes a column automatically
	
	\thanks{This work was supported by the Austrian Science Fund
		FWF projects P 27816-N26 and P 30934-N35}
	
	 %Secondary is optional
	
	\begin{abstract}
		The ring of dual numbers over a ring $R$ is
		$\Ralfa = R[x]/(x^2)$, where $\alfa$ denotes  $x+(x^2)$. 
		For any finite commutative ring $R$, we characterize null 
		polynomials and permutation polynomials on $\Ralfa$ in terms 
		of the functions induced by their coordinate polynomials
		($f_1,f_2\in R[x]$, where $f=f_1+\alpha f_2$) and their formal
		derivatives on $R$. 
		
		We derive explicit formulas for the number of polynomial functions 
		and the number of polynomial permutations on $\Zmalfa[p^n]$ for $n\le p$ ($p$ prime).
	\end{abstract}
	
	\maketitle
	
	The second and third authors wish to dedicate this paper
	to the memory of Prof.~Al-Ezeh, who died while the paper 
	was under review.

\section{Introduction}

Let $A$ be a finite commutative ring.
A function $F\colon A\longrightarrow A$ is called a polynomial function
on $A$ if there exists a polynomial $f=\sum_{k=0}^n c_k x^k\in A[x]$
such that $F(a)=\sum_{k=0}^n c_k a^k$ for all $a\in A$. When a
polynomial function $F$ is bijective, it is called a polynomial
permutation of $A$, and $f$ is called a permutation polynomial on $A$.

Polynomial functions on $A$ form a monoid $(\PolFun[A],\circ)$
with respect to  composition. Its group of units, which we denote
by $\PrPol[A]$, consists of all polynomial permutations of $A$.
Unless $A$ is a finite field, not every function on $A$ is a polynomial
function and not every permutation of $A$ is a polynomial permutation.
Apart from their intrinsic interest in algebra, polynomial functions
on finite rings have uses in computer science \cite{BuBaHo17,GuDu15}.

For any ring $R$, the ring of dual numbers over $R$ is defined as
$\Ralfa = R[x]/(x^2)$, where $x$ is an indeterminate and $\alfa$ stands
for $x+(x^2)$.  Rings of dual numbers are used in coding
theory~\cite{CeDeAy18ccc,DingLi15Gray}.

In this paper, we investigate the polynomial functions and polynomial
permutations of rings of dual numbers over finite rings.
Since every finite commutative ring is a direct sum of local rings,
and evaluation of polynomial functions factors through this direct
sum decomposition, we may concentrate on local rings.

Among other things, we derive explicit formulas for
$\left|\PolFun[\Zpnalfa]\right|$ and $\left|\PrPol[\Zpnalfa]\right|$
where  $n\le p$. Here, as in
the remainder of this paper, $p$ is a prime number and, for
any natural number $m$, $\mathbb{Z}_m$ stands for the
ring of integers modulo $m$, that is, $\mathbb{Z}_m=\mathbb{Z}/m\mathbb{Z}$.

The study of the monoid of polynomial functions and the group of
polynomial permutations on a finite ring $R$ essentially originated
with Kempner, who, in 1921, determined their orders in the case
where $R$ is the ring of integers modulo a prime power:
\begin{equation} \label{Kemp}
	\left| \PolFun[\mathbb{Z}_{p^n}] \right| =
	p^{\sum_{k=1}^n \mu(p^k)}
	\qquad\hbox{ and }\qquad \left| \PrPol[\mathbb{Z}_{p^n}]
	\right| = p! p^p (p-1)^p p^{\sum_{k=3}^n \mu(p^k)} \qquad\text{for}\  n>1,
\end{equation}
where $\mu(p^k)$ is the minimal $l\in \natn$ such that $p^k$
divides $l\/!$, that is, the minimal $l\in\natn$ for which
$\sum_{j\ge 1}\lfloor{\frac{l}{p^j}}\rfloor \ge k$.
(Here $\lfloor{z}\rfloor$ means the largest integer smaller than or equal to
$z$).

Kempner's proof has been simplified \cite{pol1,pol2,Wie83pfZm} and his
formulas shown to hold for more general classes of local
rings~\cite{gal,Suit,Nec80PTPI} when $p$ is replaced by the order of
the residue field and $n$ by the nilpotency of the maximal ideal.
The classes of local rings for which Kempner's formulas hold
{\it mutatis mutandis} have been up to now the only finite local
rings $(R,M)$ for which explicit formulas for $|\PolFun[R]|$ and
$|\PrPol[R]|$ are known. (By explicit formula, we mean one that
depends only on readily apparent parameters of the finite local ring,
such as the order of the ring and its residue field, and
the nilpotency of the maximal ideal.)

What all the finite local rings $(A,M)$ for which explicit formulas
for $|\PolFun[A]|$ and $|\PrPol[A]|$ are known have in common is
the following property:
If $m$ is the nilpotency of the maximal ideal $M$ of $A$, and we
denote by $w(a)$ the maximal $k\le m$ such that $a\in M^k$, then,
for any $a,b\in A$,
\[
w(ab) = \min(w(a)+w(b), m),
\]
that is, $A$ allows a kind of truncated discrete valuation, with
values in the additive monoid on $\{0,1,2,\ldots, m\}$, whose addition
is $u\oplus v = \min(u+v, m)$.

Rings of dual numbers over $\mathbb{Z}_{p^n}$, for which we provide
explicit formulas for the number of polynomial functions and the
number of polynomial permutations in Theorems~\ref{countfunc} and
\ref{countperm}, do not have this property, except for $n=1$, see
Proposition~\ref{sutcon}.

Statements about the number of polynomial functions and
polynomial permutations that hold for any finite commutative ring $A$ are
necessarily less explicit in nature than the counting formulas in
Equation~(\ref{Kemp}) on one hand and Theorems~\ref{countperm}
and~\ref{countfunc} on the other hand.

G\"orcs\"os, Horv\'ath and M\'esz\'aros~\cite{PERForm}
provide a formula, valid for any finite local commutative ring
that satisfies the condition $M^{|A/M|}=\{0\}$,
expressing the number of polynomial permutations in terms of the
cardinalities of the annihilators of the ideals $M_k$ generated
by the $k$-th powers of elements of the maximal ideal.
We will not make use of this formula, however, but prove our
counting formulas from scratch, in a way that yields additional
insight into the structure of the monoid of polynomial functions
and the group of polynomial permutations on rings of dual numbers.
Also for any finite local commutative ring $A$,
Jiang~\cite{ratio} has determined the ratio of
$\left|\PrPol[A]\right|$ to $\left|\PolFun[A]\right|$,
see Remark~\ref{tocorrecproof}.

Chen~\cite{Chen96pfZni}, Wei and Zhang~\cite{WZ07sos,WZ09pp2v},
Liu and Jiang~\cite{LJ09pfnv}, among others \cite{NoeOpE,weakstrong}
have generalized facts about polynomial functions in one variable
to several variables. Starting with polynomial functions over rings
of dual numbers, we get a different kind of generalization to
several parameters, if we replace $R[\alpha]$ by
$R[\alpha_1,\ldots, \alpha_n]$ with $\alpha_i\alpha_j=0$. The
second author has shown that most results of the present paper
carry over to this generalization~\cite{severalvar}.

Beyond number formulas, some structural results about groups of
permutation polynomials on $\mathbb{Z}_{p^n}$ are known, due to
N\"obauer~\cite{Noe55GRP,NoePFPR82} and others~\cite{Zh04pfpp,per2}.

In this paper,
we derive structural results about $\PolFun[\Ralfa]$ and $\PrPol[\Ralfa]$
by relating them to $\PolFun[R]$ and $\PrPol[R]$, and then use these
results to prove explicit formulas for
$|\PolFun[\Zpnalfa]|$ and $|\PrPol[\Zpnalfa]|$ in the case $n\le p$.

Here is an outline of the paper.
After establishing some notation in Section~\ref{sec1}, we
characterize  null polynomials on ${\Ralfa}$  in Section~\ref{sec2} and
permutation  polynomials on ${\Ralfa}$ in Section~\ref{sec3},
for any finite local ring $R$.
Section~\ref{stabsect} relates the pointwise stabilizer of $R$
in the group of polynomial permutations on $\Ralfa$ to functions
induced by the formal derivatives of permutation polynomials.
Section~\ref{sec4} relates permutation polynomials on $\Zmalfa[p^n]$
to permutation polynomials on $\mathbb{Z}_{p^n}$.
Section~\ref{sec5} contains counting formulas for the numbers of polynomial
functions and polynomial permutations on  $\Zmalfa[p^n]$ in terms
of the order of the pointwise stabilizer of $\mathbb{Z}_{p^n}$ in
the group of polynomial permutations on $\Zmalfa[p^n]$.
Section~\ref{sec6} contains explicit formulas for
$|\PolFun[\Zpnalfa]|$ and $|\PrPol[\Zpnalfa]|$ for $n\le p$.
Section~\ref{sec7} gives a canonical representation for
polynomial functions on $\Zmalfa[p^n]$  for $n\le p$.
The easy special case where $R$ is a finite field is treated
en passant in sections~\ref{sec2} and \ref{sec3}.

\section{Basics}\label{sec1}

We recall a few facts about rings of dual numbers and
polynomial functions, and establish our notation.
Since we are mostly concerned with polynomials over finite rings,
we have to distinguish carefully between polynomials and the
functions induced by them.      All rings are assumed to have a
unit element and to be commutative.

Throughout this paper, $p$ always stands for a prime number.
We use $\mathbb{N}$ for
the positive integers (natural numbers), $\mathbb{N}=\{1,2,3,\ldots\}$,
and $\mathbb{N}_0=\{0,1,2,\ldots\}$ for the non-negative integers.

\begin{definition}\label{polyfuncnarrow}
	Let $R$ be a ring and $a_0,\ldots,a_n\in R$.
	The polynomial $f=\sum_{i=0}^{n}a_ix^i\in R[x]$ defines (or induces)
	a function $F\colon R\longrightarrow R$ by substitution
	of the variable: $F(r)=\sum_{i=0}^{n}a_ir^i$. A function arising
	from a polynomial in this way is called a \emph{polynomial function}.
	
	If the  polynomial function $F\colon R \rightarrow R$  induced by
	$f\in R[x]$  is bijective, then $F$ is called a
	\emph{polynomial permutation} of $R$ and $f$ is called  a
	\emph{permutation polynomial} on $R$.
\end{definition}

We will frequently consider polynomials with coefficients in $\intz$
inducing functions on $\intz_m$ for various $m$. We put this on
a formal footing in the next definition.
\begin{definition}\label{Algfun}
	Let $S$ be a commutative ring, $R$ an $S$-algebra and  $f\in S[x]$.
	\begin{enumerate}
		\item
		The polynomial $f$ gives rise to a polynomial function on
		$R$, by substitution of the variable with elements of $R$. We denote
		this function by $[f]_R$, or just by $[f]$, when $R$ is understood.
		\item
		In the special case where $S=\intz$ and $R=\intz_m$, we write
		$[f]_m$ for $[f]_{\intz_m}$\!.
		\item
		When $[f]_R$ is a permutation on $R$, we call $f$ a
		\textit{permutation polynomial} on $R$.
		\item
		If $f,g\in S[x]$ such that $[f]_R=[g]_R$, we write
		$f \quv  g$ on $R$.
	\end{enumerate}
\end{definition}
\begin{remark}\label{homo}
	\leavevmode
	\begin{enumerate}
		\item
		Clearly, $\quv$ is an equivalence relation on $S[x]$.
		\item
		When $R=S$, or $R$ is a homomorphic image of $S$, the
		equivalence classes of $\quv$ are in bijective correspondence
		with the polynomial functions on $R$.
		\item
		In particular, when $R$ is finite, the number
		of different polynomial functions on $R$ equals the number of
		equivalence classes of $\quv$ on $R[x]$.
	\end{enumerate}
\end{remark}

We now introduce the class of rings whose polynomial functions
and polynomial permutations we will investigate.
\begin{definition}\label{001}
	Throughout this paper, if $R$ is a commutative ring, then $\Ralfa$
	denotes the result of adjoining $\alfa$ with $\alfa^2=0$ to $R$;
	that is, $\Ralfa$ is $R[x]/(x^2)$, where $\alfa =x+(x^2)$.
	The ring $\Ralfa$ is called the \emph{ring of dual numbers over $R$}.
\end{definition}
\begin{remark}
	Note that $R$ is canonically embedded as a subring in $\Ralfa$
	via $a\mapsto a+ 0\alpha$.
\end{remark}

For the convenience of the reader, we summarize some easy facts
about the arithmetic of rings of dual numbers.
\begin{proposition}\label{0}
	Let $R$ be a commutative ring. Then
	\begin{enumerate}
		\item[\rm (1)]
		for $a,b,c,d\in R$, we have
		
		\begin{enumerate}
			\item[\rm (a)]
			$(a+b\alfa)(c+d\alfa)=ac+(ad+bc)\alfa$
			\item[\rm (b)]
			$(a+b\alfa)$ is a unit of $\Ralfa$ if and only if
			$a$ is a unit of $R$. In this case\\
			$(a+b\alfa)^{-1}=a^{-1}-a^{-2}b\alfa$.
		\end{enumerate}
		\item[\rm (2)]
		$\Ralfa$ is a local ring if and only if $R$ is a local ring.
		\item[\rm (3)]
		If $R$ is a local ring with maximal ideal $ \mathfrak{m}$
		of  nilpotency $K$, then $\Ralfa$ is a local ring with maximal
		ideal
		$ \mathfrak{m}+\alfa R=\{a + b\alpha\mid a\in \mathfrak{m},\; b\in R\}$
		of nilpotency $K+1$.
		\item[\rm (4)]\label{resfieldiso}
		Let $(R, \mathfrak{m})$ be a local ring. The canonical embedding
		$r\mapsto r + 0\alpha$ factors through to an isomorphism of the
		residue fields of $R$ and $\Ralfa$:
		$R/\mathfrak{m} \cong \Ralfa/ (\mathfrak{m}+\alfa R)$.
	\end{enumerate}
\end{proposition}

Likewise, we summarize the  details of substituting dual numbers
for the variable in a polynomial with coefficients in the ring of
dual numbers below.

As usual, $f'$ denotes the formal derivative of a polynomial $f$.
That is, $f'=\sum_{k=1}^{n}ka_kx^{k-1}$ for $f=\sum_{k=0}^{n} a_k x^k$.
\begin{lemma}\label{02} \label{3} \label{21}
	Let $R$ be a commutative ring, and  let $a,b\in R$.
	\begin{enumerate}
		\item[\rm (1)]
		Let $f\in R[\alfa][x]$ and $f_1, f_2 \in R[x]$ be the unique polynomials
		in $R[x]$ such that $f =f_1 +\alfa f_2$. Then
		\[
		f(a+b\alfa)=f_1(a)+ (bf_1'(a)+f_2(a))\alfa.
		\]
		\item[\rm (2)]
		In the special case when $f\in R[x]$, we get
		\[
		f(a+b\alfa)=f(a)+ bf'(a)\alfa.
		\]
	\end{enumerate}
\end{lemma}

As a consequence of the above lemma, we obtain a necessary condition
for a function on $\Ralfa$ to be a polynomial function.
\begin{corollary}
	Let $F\colon \Ralfa\rightarrow \Ralfa$ such that
	$F(a+b\alfa)=c_{(a,b)}+d_{(a,b)}\alfa$
	with $c_{(a,b)}, d_{(a,b)}\in R$.
	If $F$ is a polynomial function on $\Ralfa$, then
	$c_{(a,b)}$ depends only on $a$, that is, $c_{(a,b)}=c_{(a,b_1)}$
	for all $a,b, b_1 \in R$.
\end{corollary}

The last proposition of this section goes to show that
rings of dual numbers over ${\mathbb{Z}}_{p^n}$ ($n>1$)
are a class of local rings for which no explicit formulas
for the number of polynomial functions existed previously.
By an explicit formula we mean a formula depending only on
the order of the residue field and the nilpotency of the
maximal ideal.
\begin{proposition} \label{sutcon}
	For a finite local ring $R$ with maximal ideal $\mathfrak{m}$
	of nilpotency $K$, consider the following condition:
	
	``For all $a,b\in R$ and all $k\in \mathbb{N}$,
	whenever $ab\in\mathfrak{m}^k$, it follows that
	$a\in\mathfrak{m}^i$ and $b\in \mathfrak{m}^j$ for
	$i,j\in\mathbb{N}_0$ with $i+j\geq \min(K,k)$.''
	
	Then $R=\mathbb{\mathbb{Z}}_{p^n}[\alfa]$ satisfies the condition
	if and only if $n=1$.
\end{proposition}
\begin{proof}
	Since $\mathbb{Z}_{p^n}$ is a local ring
	with maximal ideal $(p)$, $\mathbb{\mathbb{Z}}_{p^n}[\alfa]$
	is a local ring with maximal ideal
	$\mathfrak{m}=\{ap+b\alfa\mid a, b \in \mathbb{\mathbb{Z}}_{p^n}\}$
	and  $K=n+1$ by Proposition~\ref{0}. If $n=1$, then the result
	easily follows since $\mathfrak{m}^2=(0)$.
	If $n\geq 2$, then $K=n+1>2$, and $\alfa^2=0\in\mathfrak{m}^{n+1}$,
	but $\alfa \in \mathfrak{m}\smallsetminus\mathfrak{m}^2$.
\end{proof}

Local rings satisfying the condition of Proposition~\ref{sutcon}
have been called suitable in a previous paper by the third
author~\cite{Suit}. Previously known explicit formulas for the
number of polynomial functions and the number of polynomial
permutations on a finite local ring $(R,M)$ all concern suitable rings
and are the same as Kempner's formulas~(\ref{Kemp}) for
$R=\mathbb{{Z}}_{p^n}$, except that $p$ is replaced by $q=|R/M|$ and
$n$ by the nilpotency of $M$. The previous proposition shows that,
whenever $n>1$, $\mathbb{\mathbb{Z}}_{p^n}[\alfa]$ is not a
``suitable'' ring.

\section{Null polynomials on $R[\alfa]$} \label{sec2}

When one sets out to count the polynomial functions on a finite
ring $A$, one is lead to studying the ideal of
so called null-polynomials
-- polynomials in $A[x]$ that induce the zero-function on $A$ --,
because residue classes of $A[x]$ modulo this
ideal correspond bijectively to polynomial functions on~$A$.

In this section, we study null-polynomials for
rings of dual numbers $A=\Ralfa$ as defined in the previous section
(Definition~\ref{001}). We relate polynomial functions on $\Ralfa$
(induced by polynomials in $\Ralfa[x]$) to polynomial
functions induced on $\Ralfa$ by polynomials in $R[x]$, and further
to pairs of polynomial functions on $R$ arising from polynomials in
$R[x]$ and their formal derivatives.
\begin{definition} \label{020}
	Let $R$ be a commutative ring and $A$ an $R$-algebra, and notation as
	in Definition~\ref{Algfun}.
	A polynomial $f\in R[x]$ is called  a {\em null polynomial} on $A$
	if $[f]_A$ is the constant zero function, which we denote by
	$f \quv 0$  on $A$.
	
	We define $\Null$ and $\Null'$ as
	\begin{enumerate}
		\item  $\Null=\{ f\in R[x] \mid f \quv 0 \text{ on } R\}$
		\item
		$\Null'=\{ f\in R[x] \mid f \quv 0 \text{ on } R \text{ and }
		f' \quv 0 \text{ on } R\}$.
	\end{enumerate}
\end{definition}
\begin{remark}\label{nullinpolf}
	Clearly, $\Null,\Null'$ are ideals of $R[x]$, and we have
	$|\PolFun|=[R[x]:\Null]$.
\end{remark}
\begin{example}\label{insexamp}
	Let $R=\mathbb{F}_q$ be the finite field of $q$ elements. Then
	\begin{enumerate}
		\item
		$\Null[\mathbb{F}_q]=(x^q-x)\mathbb{F}_q[x]$
		\item
		$\Nulld[\mathbb{F}_q]=(x^q-x)^2\mathbb{F}_q[x]$
		\item $[\mathbb{F}_q[x]:\Nulld[\mathbb{F}_q]]=q^{2q}$.
	\end{enumerate}
	To see (2), let $g\in \Nulld[\mathbb{F}_q] $.
	Then clearly, $g(x)=h(x)(x^q-x)$.
	Hence	\[
	g'(x)=h(x)(qx^{q-1}-1)+h'(x)(x^q-x)=h'(x)(x^q-x)-h(x),
	\]
	and so $0 \quv g' \quv -h$ on $\mathbb{F}_q$.
	Thus  $h$ is a null polynomial on $\mathbb{F}_q$, and hence
	divisible by $(x^q-x)$.
\end{example}

By means of the ideal $\Null'$, we will reduce questions about
polynomials with coefficients in $\Ralfa$ to questions about
polynomials with coefficients in $R$, as exemplified in
Proposition~\ref{connection} below.

\begin{lemma}\label{31}
	Let $f\in R[x]$. Then
	\begin{enumerate}
		\item[\rm (1)]
		$f$ is a null polynomial on $\Ralfa$ if and only if
		both $f$ and $f'$ are null polynomials on $R$
		\item[\rm (2)]
		$\alfa f$ is a null polynomial on $\Ralfa$ if and only if
		$f$ is a null polynomial on $R$.
	\end{enumerate}
\end{lemma}
\begin{proof}
	Ad (1).  By Lemma~\ref{02}, for every $a,b\in R$,
	$f(a+b\alfa)=f(a)+ bf'(a)\alfa$. Thus  by Definition~\ref{020}, $f$
	being a null polynomial on $\Ralfa$ is equivalent to
	$f(a)+ bf'(a)\alfa= 0$ for all $a,b\in R$. This is equivalent to
	$f(a)=0$ and $bf'(a)= 0$ for all $a,b\in R$. Setting $b=1$, we
	see that $f(a)=0$ and $f'(a)= 0$ for all $a\in R$. Hence $f$
	and $f'$ are null polynomials on  $R$.\\
	Statement (2) follows from Lemma~\ref{02}.
\end{proof}
\begin{theorem} \label{4}
	Let $f\in \Ralfa[x]$, written as $f =f_1 + \alfa f_2$ with
	$f_1, f_2 \in R[x]$.
	\smallskip
	
	$f$ is a null polynomial on $\Ralfa$ if and only if
	$f_1$, $f_1'$, and $f_2$ are null polynomials on $R$.
\end{theorem}

\begin{proof}
	By Lemma~\ref{3}, for all $a,b\in R$,
	\[
	f(a+b\alfa)=f_1(a)+ (bf_1'(a)+f_2(a))\alfa.
	\]
	This implies the ``if'' direction.
	To  see ``only if'', suppose that $f$ is a null polynomial on
	$\Ralfa$. Then, for all $a,b\in R$,
	\begin{equation*}
		f_1(a)+ (bf_1'(a)+f_2(a))\alfa=0.
	\end{equation*}
	Clearly,  $f_1$ is a null polynomial on $R$. Substituting  $0$ for
	$b$ yields that $f_2$ is a null polynomial on $R$ and
	substituting $1$ for   $b$ yields that $f_1'$ is a null
	polynomial on $R$.
\end{proof}

Combining Lemma~\ref{31} with Theorem~\ref{4} gives the following
criterion.
\begin{corollary} \label{nullrestat}
	Let $f\in \Ralfa[x]$, written as $f =f_1 + \alfa f_2$ with
	$f_1, f_2 \in R[x]$.
	\smallskip
	
	$f$ is a null polynomial on $\Ralfa$ if and only if
	$f_1$ and $\alfa f_2$ are null polynomials on $\Ralfa$.
\end{corollary}

Also from Theorem~\ref{4}, we obtain a criterion that we will
frequently use when two polynomials induce the same polynomial
function on the ring of dual numbers.
\begin{corollary}\label{6} \label{Gencount}
	Let $f =f_1 + \alfa f_2$ and $g=g_1+\alfa g_2$, with
	$f_1, f_2, g_1, g_2 \in R[x]$.
	\smallskip
	
	$f \quv g$ on $\Ralfa$ if and only if the following three
	conditions hold:
	\begin{enumerate}
		\item[\rm (1)]
		$[f_1]_R= [g_1]_R$
		\item[\rm (2)]
		$[f_1']_R = [g_1']_R$
		\item[\rm (3)]
		$[f_2]_R =[g_2]_R$.
	\end{enumerate}
	\medskip
	
	In other words, $f \quv g$ on $\Ralfa$ if and only if
	the following two congruences hold:
	\begin{enumerate}
		\item[\rm (1)]
		$f_1 \equiv g_1  \mod \Nulld$
		\item[\rm (2)]
		$f_2 \equiv g_2  \mod \Null$.
	\end{enumerate}
\end{corollary}

We use this criterion to exhibit a polynomial with coefficients
in $R$ that induces the zero function on $R$, but not on $\Ralfa$.
\begin{example}
	Let $R=\mathbb{Z}_{p^n}$ and $n<p$. Then the polynomial
	$(x^p-x)^n$ is a null polynomial on $R$, but not on $\Ralfa$.
	Likewise, $x + (x^p-x)^n$ induces the identity function on $R$,
	but not on $\Ralfa$.
	
	To see that $x \nquv x+(x^p-x)^n$ on $R[\alfa]$, we use Corollary~\ref{6}.
	Note that
	\[
	(x+(x^p-x)^n)'=1+n(x^p-x)^{n-1}(px^{p-1}-1)
	\not\equiv 1= x' \bmod \Null.
	\]
	Hence $x\not\equiv x+(x^p-x)^n \mod \Nulld$, although
	$x\equiv x+(x^p-x)^n \mod \Null$.
	
	In a more positive vein, Corollary~\ref{6} implies that
	$x \quv x+(x^p-x)^n\alfa$ on $R[\alfa]$.
\end{example}	
\begin{remark}\label{bijection}
	Let $R$ be a finite commutative ring and $f_1, f_2\in R[x]$.  Then
	\[
	[f_1 + \alpha f_2]_{\Ralfa} \mapsto (([f_1]_R, [f_1']_R), [f_2]_R)
	\]
	establishes a well-defined bijection
	\[
	\varphi\colon \PolFun[\Ralfa] \rightarrow
	\{(G,H)\in \PolFun\times\PolFun\mid
	\exists g\in R[x] \text{ with }  G=[g]\text{ and } H=[g']\}
	\times\PolFun
	\]
	between polynomial functions on $\Ralfa$ on one hand, and triples
	of polynomial functions on $R$ such that the first two entries
	arise from a polynomial and its derivative, on the other hand.
	
	This mapping is well-defined and injective by Corollary~\ref{Gencount},
	and it is clearly onto.
\end{remark}
\begin{proposition}\label{connection}
	Let $R$ be a finite commutative ring, and let
	$\Null$ and $\Nulld$ be the ideals of Definition~\ref{020}.
	Then the number of polynomial functions on
	$\Ralfa$ is
	\[
	|\PolFun[\Ralfa]|= \big[R[x]:\Nulld \big]\big[R[x]:\Null\big].
	\]
	
	Moreover, the factors on the right have the following interpretations.
	\begin{enumerate}
		\item[\rm (1)]%%\label{pairs}
		$\left[R[x]:\Nulld\right]$ is the number of pairs of
		functions $(F,E)$ with $F\colon R\rightarrow R$,
		$E\colon R\rightarrow R$, arising as $([f], [f'])$ for some
		$f\in R[x]$.
		\item[\rm (2)]%%\label{Rinduced}
		$\left[R[x]:\Nulld\right]$ is also the number of functions
		induced on $\Ralfa$ by polynomials in $R[x]$.
		\item[\rm (3)]%%\label{obvious}
		$\left[R[x]:\Null\right]$ is the number of
		polynomial functions on $R$.
	\end{enumerate}
\end{proposition}
\begin{proof}
	Everything follows from Theorem~\ref{4}.
	In detail, consider the map $\varphi$ defined by
	\[
	\varphi\colon R[x]\times R[x] \rightarrow \mathcal{F}(R[\alfa]),
	\quad \varphi(f_1,f_2)=[f_1 + \alfa f_2 ],
	\]
	where $[f_1 + \alfa f_2 ]$ is the function
	induced on $\Ralfa$ by $f =f_1 + \alfa f_2$.
	Since every polynomial  function on $\Ralfa$ is induced by
	a polynomial $f=f_1+\alfa f_2$ with $f_1,f_2\in R[x]$,
	$\varphi$ is onto. Clearly, $\varphi$ is a homomorphism of the
	additive groups on each side.  By Theorem~\ref{4},
	$\ker\varphi=\Nulld \times \Null$. Hence, by the first
	isomorphism theorem,
	\[
	\bar\varphi\colon R[x]/\Nulld \times R[x]/\Null \rightarrow
	\mathcal{F}(R[\alfa])
	\]
	defined by $\bar\varphi(f_1 + \Nulld, f_2 + \Null) = [f_1 + \alfa f_2 ]$
	is a well defined group isomorphism.
	
	Likewise, for (1) let
	\[
	\mathcal{A}=\{(F,E)\in \mathcal{F}(R)\times \mathcal{F}(R)\mid
	\exists f\in R[x] \text{ with } [f]=F \text{ and } [f'] = E\},
	\]
	and define $\psi\colon R[x] \rightarrow \mathcal{A}$ by
	$\psi(f)=([f]_R,[f']_R)$. Then $\psi$ is a group epimorphism
	with $\ker\psi=N'_R$ and  hence $[R[x]:N'_R]=|\mathcal{A}|$.
	
	Finally, (2) follows from Corollary~\ref{6}, and
	(3) is obvious.
\end{proof}

Proposition~\ref{connection} reduces the question of counting
polynomial functions on $\Ralfa$ to determining $[R[x]:\Null]$
and $[R[x]:\Nulld]$, that is, to counting polynomial functions
on $R$ and pairs of polynomial functions on $R$ induced by a
polynomial and its derivative.
This will allow us to give explicit formulas for
$|\PolFun[\Ralfa]|$ in the case where $R=\mathbb{Z}_{p^n}$
with $n\le p$ in section~\ref{sec6}.

The simple case where $R$ is a finite field we can settle
right away by recalling from Example~\ref{insexamp} that
$\Null[\mathbb{F}_q]=(x^q-x)\mathbb{F}_q[x]$ and
$\Nulld[\mathbb{F}_q]=(x^q-x)^2\mathbb{F}_q[x]$
and hence
$[\mathbb{F}_q[x]:N'_{\mathbb{F}_q} ] = q^{2q} $
and $[\mathbb{F}_q[x]:N_{\mathbb{F}_q}] = q$.
\begin{corollary}\label{fieldfunc}
	Let $\mathbb{F}_q$ be a field with $q$ elements. Then
	$ |\PolFun[ {\mathbb{F}_q[\alpha]} ]| = q^{3q} $.
\end{corollary}

The remainder of this section is devoted to null polynomials
of minimal degree and canonical representations of polynomial
functions on $\Ralfa$ that can be derived from them.
\begin{proposition}\label{sur}
	Let $h_1\in R[\alfa][x]$ and $h_2\in R[x]$ be monic null polynomials
	on $\Ralfa$ and $R$, respectively, with $\deg h_1= d_1$
	and  $\deg h_2=d_2$.
	
	Then every polynomial function
	$F\colon\Ralfa\rightarrow \Ralfa$ is induced by a polynomial
	$f =f_1 +f_2 \alfa$ with $f_1,f_2 \in R[x]$
	such that $\deg f_1 <d_1$ and  $\deg   f_2 < \min(d_1,d_2)$.
	
	In the special case where $F$ is induced by a polynomial $f\in R[x]$
	and, also, $h_1$ is in $R[x]$, there exists a polynomial $g\in R[x]$
	with $\deg g<d_1$, such that $[g]_R=[f]_R$ and $[g']_R=[f']_R$.
\end{proposition}
\begin{proof}
	Let $g \in \Ralfa[x]$ be a polynomial that induces $F$.
	By division with remainder by $h_1$, we get
	$g(x) =q(x)h_1(x)+r(x)$ for some $r,q \in \Ralfa[x]$, where
	$\deg r < d_1$ and $r(x)$ induces $F$.
	
	We represent $r$ as $r=r_1+ \alfa r_2 $ with $r_1,r_2\in R[x]$.
	Clearly, $\deg r_1,\deg r_2 < d_1$. If $d_2<d_1$, then, we divide $r_2$
	by $h_2$ with remainder in $R[x]$ and get $f_2\in R[x]$ with
	$\deg f_2 < d_2$ and such that $f_2 \quv r_2$ on $R$.
	
	By Corollary~\ref{6}, $\alfa r_2 \quv \alfa f_2$ on $\Ralfa$ and
	hence, $f =r_1 + \alfa f_2$ has the desired properties.
	
	In the special case, the existence
	of $g\in R[x]$ with $\deg g<d_1$ such that	$f \quv g $ on
	$\Ralfa$ follows by a similar argument.
	By Corollary~\ref{Gencount}, $[g]_R=[f]_R$
	and $[g']_R=[f']_R$.
\end{proof}

In what follows, let $m,n$  be positive integers such that $m>1$ and
$p$ a prime.
\begin{definition}\label{mudef}
	For $m\in\mathbb{N}$ let $\mu(m)$ denote the smallest positive integer
	$k$ such that $m$ divides $k!$. The function
	$\mu\colon \mathbb{N}\rightarrow\mathbb{N}$ was introduced
	by  Kempner~\cite{func}.
\end{definition}

When $n\le p$, clearly $\mu(p^n)=np$.
We use this fact frequently, explicitly and sometimes implicitly.
\begin{remark}\label{01}
	It is easy to see that $m$ divides the product of any $\mu(m)$
	consecutive integers.
	
	As Kempner \cite{Residue} remarked, it follows that
	for any $c\in \mathbb{Z}$,
	\[
	(x-c)_{\mu(m)}=\prod_{j=0}^{\mu(m)-1}(x-c-j)
	\]
	is a null polynomial on $\mathbb{Z}_m$.
\end{remark}
\begin{theorem}\label{05}
	Let $m> 1$.  Then
	\begin{enumerate}
		\item[\rm (1)]
		$(x)_{2\mu(m)}$ is a null polynomial on $\Zmalfa$
		\item[\rm (2)]
		$((x)_{\mu(m)})^2$ is a null polynomial on $\Zmalfa$.
	\end{enumerate}
\end{theorem}
\begin{proof}
	Set $f(x)=(x)_{2\mu(m)}$. In view of Lemma~\ref{31}, we must
	show that $f \text{ and }  f'$ are null polynomials on
	$\mathbb{Z}_{m}$. Clearly, $f$ is a null polynomial on
	$\mathbb{Z}_m$. Now consider
	$f'(x)=\sum_{i=0}^{2\mu(m)-1}\frac{(x)_{2\mu(m)}}{x-i}$.
	Each term $\frac{(x)_{2\mu(m)}}{x-i}$ is divisible by a polynomial
	of the form $\prod_{j=0}^{\mu(m)-1}(x-c-j)$. Thus
	$\frac{(x)_{2\mu(m)}}{x-i}$ is a null polynomial on $\mathbb{Z}_m$
	by Remark \ref{01}. Hence $f'$ is a null polynomial on
	$\mathbb{Z}_{m}$. The proof of the second statement is similar.
\end{proof}

In the case when $m=p^n$, $(x)_{2\mu(p^n)}$ is a null polynomial on
$\Zmalfa[p^n]$. When $n\le p$, this says $(x)_{2np}$ is a null
polynomial on $\Zmalfa[p^n]$, but in this case more is true, namely,
$(x)_{\mu (p^n)+p}=(x)_{(n+1)p}$ is  a null polynomial on
$\Zmalfa[p^n]$.
\begin{proposition}\label{5}
	Let $n\le p$. Then $(x)_{(n+1)p}$
	is a null polynomial on $\Zmalfa[p^n]$.
\end{proposition}
\begin{proof}
	Since $n\le p$, we have $\mu (p^n)=np$. Set $f(x)=(x)_{\mu (p^n)+p}$.
	Then clearly, $f$ is a null polynomial on $\mathbb{Z}_{p^n}$.
	We represent $f(x)$ as a product of $n+1$ polynomials, each of which
	has $p$ consecutive integers as roots and is, therefore, a null-polynomial
	modulo $p$:
	\[
	(x)_{(n+1)p}=\prod_{l=0}^{n}\prod_{k=lp}^{(l+1)p-1}(x-k).
	\]
	
	Now regarding $f'(x)=\sum_{i=0}^{(n+1)p-1}\frac{(x)_{(n+1)p}}{x-i}$,
	it becomes apparent that each term $\frac{(x)_{(n+1)p}}{x-i}$ is
	divisible by a product of $n$ different polynomials of the form
	$\prod_{j=0}^{p-1}(x-c-j)$. Hence the  claim follows.
\end{proof}

Combining Theorem~\ref{05} with Proposition~\ref{sur} and Remark~\ref{01},
we obtain the following corollary, which will be needed to establish a
canonical form for a polynomial representation of a polynomial function
on $\Zmalfa[p^n]$ for $n\le p$ (see Theorems~\ref{8} and~\ref{canon}).
\begin{corollary}\label{7}
	Let $F\colon\Zmalfa\rightarrow \Zmalfa$ be a polynomial function.
	Then F can be represented as a polynomial $f \in\Zmalfa[m][x]$ with
	$ \deg f\le 2\mu(m) -1$. Moreover, $ f$ can be chosen such that
	$f =f_1 +f_2 \alfa $, 	with $f_1,f_2\in \mathbb{Z}_m[x]$,
	$\deg f_1\le 2\mu(m) -1$ and $\deg f_2 \le \mu(m) -1$.
\end{corollary}

When $R=\mathbb{F}_q$ is a finite field, we have already remarked
in Corollary~\ref{fieldfunc}
that the number of polynomial functions on $\mathbb{F}_q[\alfa]$
is $q^{3q}$.
We can make this more explicit by giving a canonical
representation for the different  polynomial functions on
$\mathbb{F}_q[\alfa]$.
\begin{corollary}\label{08.1}
	Let $\mathbb{F}_q$ be a finite field with $q$ elements.
	Every polynomial function
	$F\colon \mathbb{F}_q[\alfa]\rightarrow \mathbb{F}_q[\alfa]$
	can be represented uniquely  as a polynomial
	\begin{equation}\label{reps}
		f(x)= \sum_{i=0}^{2q-1}a_ix^i+
		\sum_{j=0}^{q-1}b_jx^j\alfa\qquad
		\text{ for }a_i,b_j\in \mathbb{F}_q.
	\end{equation}
\end{corollary}
\begin{proof}
	We note that the polynomials $(x^q-x)^2$ and $(x^q-x)$ satisfy the
	conditions of Proposition~\ref{sur}. Thus every polynomial function
	on $\mathbb{F}_q[\alfa]$ is represented by a polynomial as in
	Equation~(\ref{reps}).	
	
	Since there are exactly $q^{3q}$ different polynomials of the form
	(\ref{reps}) and also, by Corollary~\ref{fieldfunc}, $q^{3q}$ different
	polynomial functions on $\mathbb{F}_q[\alfa]$, every polynomial function is
	represented uniquely.
	
	We can also show uniqueness directly, without using
	Corollary~\ref{fieldfunc}, by demonstrating that every expression of
	type~(\ref{reps}) representing the zero function is the zero polynomial.
	Let $f\in \mathbb{F}_q[\alfa][x]$ be a null polynomial on
	$\mathbb{F}_q[\alfa]$ with
	$ f(x)=\sum_{i=0}^{2q-1}a_ix^i+ \sum_{j=0}^{q-1}b_jx^j\alfa$.
	
	Then
	$\sum_{i=0}^{2q-1}a_ix^i  \in \Nulld[\mathbb{F}_q] $ and
	$\sum_{j=0}^{ q-1}b_jx^j  \in \Null[\mathbb{F}_q]$ by
	Theorem~\ref{4}. Recalling from Example~\ref{insexamp} that
	$\Nulld[\mathbb{F}_q]=(x^q-x)^2\mathbb{F}_q[x]$ and
	$\Null[\mathbb{F}_q]=(x^q-x)\mathbb{F}_q[x]$, we see that $a_i=0$
	for $i=0,\ldots,2q-1$; and $b_j=0$ for  $j=0,\ldots,q-1$.
\end{proof}

\section{Permutation polynomials on $\Ralfa$}\label{sec3}

We know direct our attention to permutation polynomials on $\Ralfa$,
where $\Ralfa$ is the ring of dual numbers over a finite commutative
ring $R$ (defined in Definition~\ref{001}). As in the previous section,
we first relate properties of polynomials in $\Ralfa[x]$ to properties
of polynomials in $R[x]$, about which more may be known.
\begin{theorem}\label{Genper}
	Let  $R$ be a commutative ring. Let $f =f_1 +\alfa f_2$,
	where $f_1,f_2 \in R[x]$. Then $f$ is a permutation polynomial
	on $R[\alfa]$ if and only if  the following conditions hold:
	\begin{enumerate}
		\item[\rm (1)]
		$f_1$ is a permutation polynomial on $R$
		\item[\rm (2)]
		for all $a\in R$, $f_1'(a)$ is a unit of $R$.
	\end{enumerate}
\end{theorem}
\begin{proof}
	$(\Rightarrow)$
	To see (1), let $c\in R$. Since $f$ is a
	permutation polynomial on $ R[\alfa]$, there exist
	$a,b\in R$ such that $c=f(a+b\alfa)$, that is,
	$c=f_1(a)+(bf_1'(a)+f_2(a))\alfa$ (by Lemma~\ref{3}).
	In particular, $f_1(a)=c$, and, therefore, $[f_1]_R$ is onto and
	hence a permutation of $R$.
	
	To see (2), let $a\in R$ and suppose that $f_1'(a)$ is not a unit of
	$R$. $R$ being finite, it follows that $f_1'(a)$ is a zerodivisor of $R$.
	Let $b\in R$, $b\ne 0$, such that $bf_1'(a)=0$. Then
	\[
	f(a+b\alfa)=f_1(a)+(bf_1'(a)+f_2(a))\alfa =f_1(a)+f_2(a)\alfa=f(a).
	\]
	So $f$ is not one-to-one; a contradiction.
	
	($\Leftarrow$) Assume (1) and (2) hold.
	It suffices to show that $[f]_{\Ralfa}$ is one-to-one.
	Let $a,b,c,d \in R$ such that
	$f(a+b\alfa)=f(c+d\alfa)$, that is,
	\[
	f_1(a)+(bf_1'(a)+f_2(a))\alfa = f_1(c)+(df_1'(c)+f_2(c))\alfa.
	\]
	Then $f_1(a)= f_1(c)$ and hence $a= c$, by (1).
	Furthermore, $bf_1'(a)= d\/f_1'(a)$, and, since $f_1'(a)$ is
	not a zerodivisor, $b=d$ follows.
\end{proof}

The special case of polynomials with coefficients in $R$ is
so important that we state it separately.

We call a function on $R$ that maps every element of $R$ to a
unit of $R$ a {\em unit-valued} function on~$R$.
\begin{corollary}\label{RcoeffGenper}
	Let  $R$ be a commutative ring and $f\in R[x]$.
	Then $f$ is a permutation polynomial
	on $R[\alfa]$ if and only if the following two conditions hold:
	\begin{enumerate}
		\item[\rm (1)]
		$[f\/]_R$ is a permutation of $R$
		\item[\rm (2)]
		$[f']_R$ is unit-valued.
	\end{enumerate}
\end{corollary}

Theorem~\ref{Genper} shows that whether $f=f_1+\alfa f_2\in\Ralfa[x]$
is a permutation polynomial on $R[\alfa]$ depends only on $f_1$.
In particular, $f_1 + \alfa f_2$ is a permutation polynomial
on $R[\alfa]$ if and only if $f_1+\alfa\cdot 0$ is a permutation
polynomial on $R[\alfa]$. We rephrase the last remark as a
corollary.
\begin{corollary}\label{PPfirstcoordinate}
	Let  $R$ be a finite ring. Let $f =f_1 + \alfa f_2$,	where
	$f_1,f_2 \in R[x]$. Then $f$ is a permutation polynomial
	on $R[\alfa]$ if and only if $f_1$ is a permutation polynomial
	on $R[\alfa]$.
\end{corollary}
\begin{corollary} \label{Geperncount}
	Let $R$ be a finite ring and $R^*$ the group of units on $R$.
	Let $B$ denote  the number of pairs of	functions $(H,G)$ with
	\[
	H\colon R\rightarrow R \text{ bijective\qquad  and }\qquad
	G\colon R\rightarrow R^*
	\]
	that occur as $([g],[g'])$ for some $g\in R[x]$. Then  the number
	$|\PrPol[\Ralfa]|$ of polynomial permutations on $\Ralfa$  is equal
	to
	\[
	|\PrPol[\Ralfa]|=B\cdot |\PolFun|.
	\]
\end{corollary}
\begin{proof}
	By Corollary~\ref{Gencount} and Remark~\ref{bijection},
	\[
	[f_1 +\alpha f_2]_{\Ralfa} \mapsto ([f_1]_R,[f_1']_R,[f_2]_R)
	\]
	is a bijection between $\PolFun[\Ralfa]$ and triples of polynomial
	functions on $R$ such that the first two entries of the triple arise
	from one polynomial and its derivative.
	
	By Theorem~\ref{Genper}, the restriction of this bijection to
	$\PrPol[\Ralfa]$ is surjective onto the set of those triples
	$([f_1]_R,[f_1']_R,[f_2]_R)$ such that $[f_1]_R$ is bijective and
	$[f_1']_R$ takes values in $R^*$.
\end{proof}

We now introduce a subgroup of the group of polynomial permutations
of a ring of dual numbers that will play an important role in
determining the order of the group.
\begin{definition}\label{std}
	Let
	\[St_{\alfa}(R)=\{F\in \PrPol[\Ralfa]\mid
	F(a)=a \text{ for every } a\in R\}.
	\]
	$St_{\alfa}(R)$, which is clearly a subgroup of $\PrPol[\Ralfa]$,
	is called the pointwise stabilizer (or shortly the stabilizer)
	of $R$ in the group $\mathcal{P}(\Ralfa)$.
\end{definition}
\begin{proposition}\label{firststab}
	Let $R$ be a finite commutative ring. Then
	\[
	St_{\alfa}(R)=\{F\in \PrPol[\Ralfa] \mid
	F \textnormal{ is induced by } x+h(x), \textnormal{ for some }
	h \in \Null[R] \}.
	\]
	In particular, every element of the stabilizer of $R$ can be realized
	by a polynomial in $R[x]$.
\end{proposition}
\begin{proof}
	It is clear that
	\[
	St_{\alfa}(R)\supseteq\{F\in \PrPol[\Ralfa] \mid F
	\textnormal{ is induced by } x+h(x), \textnormal{ for some }
	h \in \Null[R] \}.
	\]
	Now, let $F\in \mathcal{P}(\Ralfa)$
	such that $F(a)=a$ for every $a\in R$.
	Then $F$ is represented by $f_1+f_2\alfa$,
	where $f_1, f_2\in R[x]$, and
	$a=F(a)=f_1(a)+f_2(a)\alfa$ for every $a\in R$.
	It follows that $f_2(a)=0$ for every $a\in R$, i.e.,
	$f_2$ is a null polynomial on $R$.
	Thus, $f_1+f_2\alfa\quv f_1$ on $\Ralfa$
	by Lemma~\ref{3}, that is, $F$ is represented by $f_1$.
	Therefore,
	$[f_1]_R=id_{R}$ (since $F$ is the identity on $R$) and, so,
	$f_1(x)=x+h(x)$ for some $h\in R[x]$ that is a
	null polynomial on $R$.
\end{proof}

\begin{remark}
	To prevent confusion about the expression for the
	stabilizer group in Proposition~\ref{firststab} we emphasize
	that, in general, not every polynomial of the form $x+h$ with
	$h \in \Null[R]$ induces  a polynomial permutation of $R[\alfa]$,
	as the following example shows.
\end{remark}
\begin{example}
	Let $R=\mathbb{F}_q$. Consider the polynomial
	$(x^q-x)\in \Null[\mathbb{F}_q] $. 	Then the polynomial
	$f(x)=x+(x^q-x)=x^q$ induces the identity on $\mathbb{F}_q$,
	but $f$	is not a permutation polynomial on $\mathbb{F}_q[\alfa]$,
	since $f(\alfa)=f(0)=0$. Thus $f$ does not
	induce an element of $St_{\alfa}(\mathbb{F}_q)$.
\end{example}

The remainder of this section is concerned with polynomial
permutations of the ring of dual numbers in the simple case where
the base ring is a finite field.
We already determined the number of polynomial functions on the
dual ring over a finite field (see Corollary~\ref{fieldfunc}).
The number of polynomial permutations
now follows readily from Corollary~\ref{Geperncount}, since every
pair of functions on a finite field arises as the pair of functions
induced by a polynomial and its derivative.

\begin{lemma}\label{Perf}
	Let $\Fq$ be a finite field with $q$ elements.
	Then for all functions
	$F,G\colon\mathbb{F}_q\rightarrow \mathbb{F}_q $
	there exists a polynomial $ f\in\mathbb{F}_q[x]$ such that
	\[(F, G)=([f\/],[f'\/]) \quad\text{and}\quad \deg f<2q.\]
\end{lemma}
\begin{proof}
	Let $f_0,f_1\in\mathbb{F}_q[x]$ such that $[f_0] =F$
	and $[f_1] =G$ and set
	\[
	f(x) = f_0(x) + (f'_0(x) - f_1(x))(x^q-x).
	\]
	Then $[f]=[f_0]=F$ and $[f']=[f_1]=G$. Moreover,
	by division with remainder by $(x^q-x)$,
	we can find $f_0, f_1$ such that $\deg f_0,\deg f_1<q$.
\end{proof}
\begin{proposition}\label{13.111}
	Let $\mathbb{F}_q$ be a finite  field with $q$ elements.
	The number  $|\mathcal{P}(\mathbb{F}_q[\alfa])|$ of polynomial
	permutations on $\mathbb{F}_q[\alfa]$  is given by
	\[
	|\mathcal{P}(\mathbb{F}_q[\alfa])|=q!(q-1)^q q^q.
	\]
\end{proposition}
\begin{proof}
	Let $\mathcal{B}$ be the set of pairs of functions $(H,G)$ such that
	\[
	H\colon \mathbb{F}_q\rightarrow \mathbb{F}_q  \text{ bijective\qquad and }\qquad
	G\colon \mathbb{F}_q\rightarrow \mathbb{F}_q\smallsetminus\{0\}.
	\]
	Clearly, $|\mathcal{B}|= q!(q-1)^q$.
	By Lemma~\ref{Perf}, each $(H,G)\in \mathcal{B}$ arises as	
	$([f],[f'])$ for some $f\in \mathbb{F}_q[x]$. Thus by
	Corollary~\ref{Geperncount},
	$|\mathcal{P}(\mathbb{F}_q[\alfa])|=
	|\mathcal{B}|\cdot |\PolFun[\mathbb{F}_q]|=
	q!(q-1)^q q^q$.
\end{proof}

When $R$ is a finite field, then, as we have seen, we do not need the
stabilizer group to determine the number of polynomial permutations
on the ring of dual numbers. We will nevertheless investigate this
group, starting with its order, for comparison purposes, and because
it yields some information on the structure of $\PrPol[\Fqa]$.
\begin{theorem} \label{1301}
	Let $\mathbb{F}_q$ be a finite a field with $q$ elements. Then
	\begin{enumerate}
		\item[\rm (1)]
		$|St_{\alfa}(\mathbb{F}_q)| =|\{[f'\/]_{\mathbb{F}_q}\mid
		f\in \mathbb{F}_q[x],\; [f\/]_{\mathbb{F}_q} = id_{\mathbb{F}_q}
		\text{ and }[f'\/]_{\mathbb{F}_q}\text{ is unit-valued} \}|$
		\item[\rm (2)]
		$|St_{\alfa}(\mathbb{F}_q)|=
		|\{[f'\/]_{\mathbb{F}_q}\mid f\in \mathbb{F}_q[x],\
		[f\/]_{\mathbb{F}_q} = id_{\mathbb{F}_q},\;
		\deg f<2q \text{ and } [f'\/]_{\mathbb{F}_q}\text{ is unit-valued} \}|$
		\item[\rm (3)]
		$|St_{\alfa}(\mathbb{F}_q)|  =(q-1)^q$. %\label{st1}
	\end{enumerate}
\end{theorem}
\begin{proof}
	To see (1), set
	$A=\{[f'\/]_{\mathbb{F}_q}\mid f\in \mathbb{F}_q[x],\;
	[f\/]_{\mathbb{F}_q} = id_{\mathbb{F}_q}
	\text{ and }[f'\/]_{\mathbb{F}_q}\text{ is unit-valued}\}$.
	We define a bijection $\varphi$ from
	$St_{\alfa}(\mathbb{F}_q)$  to $A$.
	Given $F\in St_{\alfa}(\mathbb{F}_q)$,
	there exists a polynomial $f\in \mathbb{F}_q[x]$ inducing $F$
	on $\mathbb{F}_q[\alfa]$
	such that $[f\/]_{\mathbb{F}_q}=id_{\mathbb{F}_q}$ by Definition~\ref{std}.
	By Theorem~\ref{Genper}, $[f'\/]_{\mathbb{F}_q}$ is unit-valued.
	We set $\varphi(F)=[f'\/]_{\mathbb{F}_q}$.
	Corollary~\ref{6} shows that $\varphi$ is well-defined and injective,
	and Theorem~\ref{Genper} shows that it is surjective.
	
	(2) follows from (1) and Lemma~\ref{Perf}.
	Ad (3). By (1),
	$|St_{\alfa}(\mathbb{F}_q)|\le|\{G \colon \mathbb{F}_q
	\rightarrow \mathbb{F}_q^* \}|= (q-1)^q.$
	Now consider a function
	$G \colon \mathbb{F}_q\rightarrow \mathbb{F}_q^*$.
	By Lemma~\ref{Perf}, there exists a polynomial  $h\in \mathbb{F}_q[x]$
	such that $[h]_{\mathbb{F}_q}=id_{\mathbb{F}_q}$ and
	$[h']_{\mathbb{F}_q}=G$.
	Thus $h$ represents an element of $St_{\alfa}(\mathbb{F}_q)$,
	and $G$  maps to this element under the bijection $\varphi$ in the
	proof of (1).  Hence $| St_{\alfa}(\mathbb{F}_q)|\ge (q-1)^q$.
\end{proof}

The equalities of Theorem~\ref{1301} actually come from a
group isomorphism, as the second author has  shown~\cite{unitval}.
By Proposition~\ref{13.111} and Theorem~\ref{1301},
we immediately see the special case for finite fields of a
more general result that we will show in the next
section~(see Theorem~\ref{14}).
\begin{corollary}\label{countpn=1}
	The number  $|\mathcal{P}(\mathbb{F}_q[\alfa])|$ of polynomial
	permutations on $\mathbb{F}_q[\alfa]$  is given by
	\[
	|\mathcal{P}(\mathbb{F}_q[\alfa])|=|\mathcal{P}(\mathbb{F}_q)|
	|\mathcal{F}(\mathbb{F}_q)||St_{\alfa}(\mathbb{F}_q)|.
	\]
\end{corollary}

\section{The stabilizer of $R$\\ in the group of polynomial
	permutations of $\Ralfa$ }\label{stabsect}

In this section we express the numbers of polynomial functions
and polynomial permutations on $\Ralfa$ in terms of the order of
$St_{\alfa}(R)$, the stabilizer of $R$, that is, the group of
those polynomial permutations of $\Ralfa$ that fix $R$ pointwise.
The group of those polynomial permutations of $\Ralfa$ that can
be realized by polynomials with coefficients in $R$ will play
a role, as it contains the stabilizer.

\begin{Notation}\label{freealfper}
	Let
	$
	\mathcal{P}_R(R[\alfa]) =
	\{F\in \PrPol[{R[\alfa]}]\mid F=[f\/] \textnormal{ for some }
	f\in R[x]\}
	$.
\end{Notation}
\begin{remark}
	Proposition~\ref{firststab} shows that the elements of $St_{\alfa}(R)$,
	a priori induced by polynomials in $\Ralfa[x]$,
	can be realized by polynomials in $R[x]$, that is,
	\[ St_{\alfa}(R)\subseteq \mathcal{P}_R(R[\alfa]). \]
\end{remark}

The following well-known, useful characterization of permutation
polynomials on finite local rings has been shown by
N\"obauer~\cite[section~III, statement~6, pp.~335]{Noe64PTPP} (also
for several variables \cite[Theorem~2.3]{Noe64PTPP}). It is
implicitly shown in the proof of a different result in McDonalds's
monograph on finite rings~\cite[pp.~269--272]{finiterings}, and
explicitly in a paper of Nechaev~\cite[Theorem~3]{Nec80PTPI}.
\begin{lemma}[{\cite[Theorem.~2.3]{Noe64PTPP}}]\label{Necha}
	Let $R$ be a finite local ring, not a field, $M$ its maximal ideal, and
	$f\in R[x]$.
	
	Then $f$ is a permutation polynomial on $ R$
	if and only if the following conditions hold:
	\begin{enumerate}
		\item[\rm (1)]
		$f$ is a permutation polynomial on $R/M$
		\item[\rm (2)]
		for all $a\in R$, $f'(a)\ne 0\mod{M}$.
	\end{enumerate}
\end{lemma}
\begin{lemma}\label{forunitdreiv}
	Let $R$ be a finite commutative ring and $F\in \PrPol$.
	Then there exists a polynomial $f\in R[x]$ such that $[f]_R=F$
	and $f'(r)$ is a unit of $R$ for every $r\in R$.
\end{lemma}
\begin{proof}
	Since every finite commutative ring is a direct sum of local rings,
	we may assume $R$ local.  When  $R$ is a finite field,
	the statement follows from Lemma~\ref{Perf}, while,
	when  $R$ is a finite local ring but not a field,
	it follows from Lemma~\ref{Necha}.
\end{proof}
\begin{lemma}\label{PZlemma}
	$\mathcal{P}_R(\Ralfa)$  is a subgroup of $\PrPol[\Ralfa]$; and the map
	\[
	\varphi\colon \mathcal{P}_R(\Ralfa) \rightarrow \PrPol[R]
	\quad\text{defined by}\quad
	F\mapsto \restrict{F}{R}
	\quad\text{(the restriction of $F$ to $R$)}
	\]
	is a group epimorphism with $\ker{\varphi}=St_{\alfa}(R)$.
	In particular,
	\begin{enumerate}
		\item[\rm (1)]%%\label{firstPZ}
		every element of $\PrPol[R]$ occurs as
		the restriction to $R$ of some $F\in \mathcal{P}_R(\Ralfa)$
		\item[\rm (2)]%%\label{scndPZ}
		$\mathcal{P}_R(\Ralfa)$ contains
		$St_{\alfa}(R)$ as a normal subgroup and
		\[
		\raise2pt\hbox{$\mathcal{P}_R(\Ralfa)$} \big/
		\lower2pt\hbox{$St_{\alfa}(R)$}
		\;\cong\; \PrPol[R].
		\]
	\end{enumerate}
\end{lemma}
\begin{proof}
	$\mathcal{P}_R(\Ralfa)$ is a finite subset of $\PrPol[R]$ that
	is closed under composition, and hence a subgroup of $\PrPol[R]$.
	Polynomial permutations of $\Ralfa$ induced by polynomials in $R[x]$
	map $R$ to itself bijectively. The map $\varphi$ is therefore well defined,
	and clearly a homomorphism with respect to composition of functions.
	
	Ad (1) This is evident from Theorem~\ref{Genper} and
	Lemma~\ref{forunitdreiv}.
	
	Ad (2)
	$St_{\alfa}(R)$ is contained in $\mathcal{P}_R(\Ralfa)$,
	by Proposition~\ref{firststab}.
	$St_{\alfa}(R)$, the pointwise
	stabilizer of $R$ in $\PrPol[\Ralfa]$ is, therefore, equal to
	the pointwise stabilizer of $R$ in $\mathcal{P}_R(\Ralfa)$,
	which is the kernel of $\varphi$.
\end{proof}

Recall that a function on $R$ is unit-valued if it maps $R$
into, $R^*$\!, the group of units on $R$.
\begin{corollary} \label{mshi}
	For any fixed $F\in \PrPol[R]$,
	\[
	\left|St_{\alfa}(R)\right|=
	\left|\{([f]_R,[f']_R)\mid
	f\in R[x],\qquad
	[f]_{R}= F, \text{ and\/ }
	[f']_R \text{ is unit-valued}\}\right|.
	\]
\end{corollary}
\begin{proof}
	Let $f\in R[x]$ such that $[f]_{R}= F$ and $[f']_R$ is unit-valued.
	Such a polynomial $f$ exists by Lemma~\ref{forunitdreiv}.
	By Corollary~\ref{RcoeffGenper}, $f$ induces a permutation of $\Ralfa$,
	which we denote by $[f]$.
	
	Let $C$ be the coset of $[f]$ with respect to $St_{\alfa}(R)$.
	Then $\left|C\right|= |St_{\alfa}(R)|$.
	By Lemma~\ref{PZlemma} (2),  $C$ consists precisely
	of those polynomial permutations $G\in \mathcal{P}_R(\Ralfa)$ with
	$\restrict{G}{R}=F$.
	
	A bijection $\psi$ between $C$ on one hand and
	the set of pairs $([g]_{R},[g']_{R})$, where $g\in R[x]$ such that
	$[g]_R= F$ and $[g']_R$ is unit-valued on the other hand
	is given by $\psi(G)=([g]_{R},[g']_{R})$,
	where $g$ is any polynomial in $R[x]$ which induces $G$ on $\Ralfa$.
	The map $\psi$ is well-defined and injective by Corollary~\ref{Gencount}
	and onto by Corollary~\ref{RcoeffGenper}.
\end{proof}
\begin{theorem}\label{14}
	Let $R$ be a finite local ring. Then
	\[
	|\PrPol[\Ralfa]|=|\PolFun[{R}]|\cdot
	|\PrPol[R]|\cdot |St_{\alfa}(R)|.
	\]
\end{theorem}
\begin{proof}
	Set
	\[
	B = \{([f]_R,[f']_R)\mid f\in R[x],\; [f]_R\in P(R) \text{ and }
	[f']_R \text{ is unit-valued}  \}.
	\]
	By Corollary~\ref{mshi},
	$|B|=|\PrPol[R]|\cdot |St_{\alfa}(R)|$.
	
	We define a function
	$\psi\colon \PrPol[\Ralfa] \rightarrow
	B\times \PolFun[{R}]$
	as follows: if $G\in\PrPol[\Ralfa]$ is induced by
	$g=g_1+\alfa g_2$, where $g_1,g_2 \in R[x]$,
	we let $\psi(G)=(([g_1]_R,[g'_1]_R),[g_2]_R)$.
	By Theorem~\ref{Genper} and Corollary~\ref{6}, $\psi$ is
	well-defined and one-to-one. The surjectivity of $\psi$ follows by
	Theorem~\ref{Genper}. Therefore,
	\begin{equation*}
		|\PrPol[\Ralfa]|=|B\times \PolFun[{R}]|=
		|\PrPol[R]|\cdot |St_{\alfa}(R)|\cdot
		| \PolFun[{R}]|.\qedhere
	\end{equation*}
\end{proof}
\begin{remark}\label{tocorrecproof}
	Let $R$ be a finite local ring which is not a field, $M$ the maximal
	ideal of $R$, and $q=\left| R/M \right|$.
	Jiang~\cite{ratio} has shown the following relation between the number of
	polynomial functions and the number of polynomial permutations on $R$:
	
	\[
	|\PrPol[R]|=\frac{q!(q-1)^q}{q^{2q}}|\PolFun[R]|.
	\]
\end{remark}
\begin{corollary}\label{14a}
	Let $R$ be a finite local ring which is not a field. Then
	\[|\PolFun[\Ralfa]|=|\PolFun[{R}]|^2\cdot |St_{\alfa}(R)|.\]
\end{corollary}
\begin{proof}
	The residue fields of $R$ and $\Ralfa$ are isomorphic by
	Proposition~\ref{0}~(4). Let $q$ denote the order
	of this residue field.
	By Theorem~\ref{14},
	$|\PrPol[\Ralfa]|=|\PolFun[{R}]|\cdot |\PrPol[R]|\cdot |St_{\alfa}(R)|$.
	Now apply Remark~\ref{tocorrecproof} to  $\PrPol[\Ralfa]$
	and $\PrPol[R]$ simultaneously and cancel.
\end{proof}

\section{Permutation polynomials on $\Zmalfa$}\label{sec4}

In this section we characterize permutation polynomials on
$\mathbb{Z}_{p^n}[\alfa]$ in relation to permutation polynomials
on $\mathbb{Z}_{p^n}$.
\begin{lemma}[{\cite[Hilfssatz 8]{Noe53GRP}}]\label{10.1}
	Let $n>1$, and $f\in \mathbb{Z}[x]$.
	Then $f$ is a permutation polynomial on $\mathbb{Z}_{p^n}$
	if and only if the following conditions hold:
	\begin{enumerate}
		\item[\rm (1)]
		$f$ is a permutation polynomial on $\mathbb{Z}_{p}$
		\item[\rm (2)]
		for all $a\in \mathbb{Z}$, $f'(a)\not\equiv 0\pmod{p}$.
	\end{enumerate}
\end{lemma}

We now apply the principle of Lemma~\ref{10.1} to
Theorem~\ref{Genper} and Corollary~\ref{PPfirstcoordinate} in the
special case where $R=\mathbb{Z}_{p^n}$.
\begin{theorem}\label{PPalfaequivalence}\label{11}%%\label{10}\label{11a}
	Let $f\in\mathbb{Z}[\alfa][x]$, $f=f_1+\alpha f_2$ with
	$f_1,f_2\in\mathbb{Z}[x]$. Then the following statements are equivalent:
	\begin{enumerate}
		\item[\rm (1)]%%\label{PPallnalfa}
		$f$ is a permutation polynomial on $\Zmalfa[p^n]$ for all $n\ge 1$
		\item[\rm (2)]%\label{PPsomenalfa}
		$f$ is a permutation polynomial on $\Zmalfa[p^n]$ for some $n\ge 1$
		\item[\rm (3)]%%\label{PPoneallnalfa}
		$f_1$ is a permutation polynomial on $\Zmalfa[p^n]$ for all $n\ge 1$
		\item[\rm (4)]%%\label{PPonesomenalfa}
		$f_1$ is a permutation polynomial on $\Zmalfa[p^n]$ for some $n\ge 1$
		\item[\rm (5)]%%\label{PPmodpderivative}
		$f_1$ is a permutation polynomial on $\mathbb{Z}_p$  and  for all
		$a\in\mathbb{Z}$, $f_1'(a)\not\equiv 0\pmod{p}$
		\item[\rm (6)]%%\label{PPalln}
		$f_1$ is a permutation polynomial on $\mathbb{Z}_{p^n}$ for all $n\ge 1$
		\item[\rm (7)]%%\label{PPsomen}
		$f_1$ is a permutation polynomial on $\mathbb{Z}_{p^n}$ for some $n> 1$.
	\end{enumerate}
\end{theorem}
\begin{proof}
	By Corollary~\ref{PPfirstcoordinate},
	(1) is equivalent to (3), and
	(2) is equivalent to (4).
	By Lemma~\ref{10.1}, the statements
	(5), (6) and (7)
	are equivalent.
	
	By Theorem~\ref{Genper}, (1) is equivalent to
	(6) together with the fact that
	$f_1'(a)\not\equiv 0\pmod{p}$ 	for any $a\in \mathbb{Z}$.
	But Lemma~\ref{10.1} shows that the condition	on the derivative
	of $f_1$ is redundant. Therefore, (1) is
	equivalent to (6).
	
	(1) implies (2) a fortiori.
	Finally, taking into account the fact that a permutation
	polynomial on $\mathbb{Z}_{p^n}$ is also a permutation polynomial
	on $\mathbb{Z}_p$, Theorem~\ref{Genper} shows	that (2)
	implies (5).
\end{proof}

The special case $f=f_1$ yields the following corollary.
\begin{corollary}\label{PPconstantcoeffequivalence}
	Let $f\in\mathbb{Z}[x]$. Then the following statements are equivalent:
	\begin{enumerate}
		\item[\rm (1)]%%\label{allnalfa}
		$f$ is a permutation polynomial on $\Zmalfa[p^n]$ for all $n\ge 1$
		\item[\rm (2)]%%\label{somenalfa}
		$f$ is a permutation polynomial on $\Zmalfa[p^n]$ for some $n\ge 1$
		\item[\rm (3)]%%\label{modpderivative}
		$f$ is a permutation polynomial on $\mathbb{Z}_p$ and for all
		$a\in\mathbb{Z}$, $f'(a)\not\equiv 0\pmod{p}$
		\item[\rm (4)]%%\label{alln}
		$f$ is a permutation polynomial on $\mathbb{Z}_{p^n}$ for all $n\ge 1$
		\item[\rm (5)]%%\label{somen}
		$f$ is a permutation polynomial on $\mathbb{Z}_{p^n}$ for some $n> 1$.
	\end{enumerate}
\end{corollary}

We exploit the equivalence of being a permutation polynomial on
$\Zmalfa[p^n]$ and being a permutation polynomial on
$\mathbb{Z}_{p^n}$ (only valid for $n>1$) in the following corollary,
always keeping in mind that being a null-polynomial on $\mathbb{Z}_{p^n}$
is not equivalent to being a null-polynomial on $\Zmalfa[p^n]$.
\begin{corollary}\label{corof11}
	Let $n>1$, and $f,g\in \mathbb{Z}[x]$.
	\begin{enumerate}
		\item[\rm (1)]
		If $f$ is a permutation polynomial on $\mathbb{Z}_{p^n}$ and $g$
		a null polynomial on $\mathbb{Z}_{p^n}$ then $f+g$
		is a permutation polynomial on $\Zmalfa[p^n]$.
		\item[\rm (2)]
		In particular, if $g$ is a null-polynomial on $\mathbb{Z}_{p^n}$,
		$x+g$ induces an element of $\Stab[{p^n}]$.
	\end{enumerate}
\end{corollary}
\begin{proof}
	Ad (1).
	Set $h=f+g$. Then $[h]_{p^n} = [f]_{p^n}$ and $h$ is, therefore,
	a permutation polynomial on $\mathbb{Z}_{p^n}$.
	Since $n>1$, Corollary~\ref{PPconstantcoeffequivalence} applies
	and $h(x)$ is a permutation polynomial on $\Zmalfa[p^n]$.
	Now (2) follows from (1) and Definition~\ref{std}.
\end{proof}

The following example illustrates the necessity of the condition $n>1$
in Theorem~\ref{11}~(7) and Corollary~\ref{corof11}.
\begin{example}
	Consider the polynomials $f(x)=(p-1)x$ and
	$g(x)=(p-1)(x^p-x)$. Clearly, $f$ is a permutation polynomial
	on both $\mathbb{Z}_p$ and $\mathbb{Z}_p[\alfa]$,
	while $g(x)$ is a null polynomial on $\mathbb{Z}_p$.
	Now, $h(x)=f(x)+g(x)=(p-1)x^p$ permutes the
	elements of $\mathbb{Z}_p$, but $h$ is not a permutation polynomial
	on $\mathbb{Z}_p[\alfa]$, as $h(\alfa)=h(0)=0$.
\end{example}

We can apply the Chinese Remainder Theorem to Theorem~\ref{11}
and Corollary~\ref{corof11} to obtain statements about permutation
polynomials on $\mathbb{Z}_m[\alfa]$.

\begin{theorem}\label{11.1}
	Let $f =f_1 +\alfa f_2$ with $f_1, f_2 \in \mathbb{Z}[x]$.
	Then $f$ is a permutation polynomial on $\mathbb{Z}_m[\alfa]$
	if and only if for every prime  $p$ dividing $m$,
	$f_1$ is a permutation polynomial on $\mathbb{Z}_{p}$ and
	$f_1'$ has no zero modulo $p$.
\end{theorem}
\begin{corollary}\label{11.2}
	Let  $m=p_1^{n_1}\cdots p_k^{n_k}$, where $p_1,\ldots,p_k$ are
	distinct primes and $n_j>1$ for $j=1,\ldots,k$.
	Let $f,g\in \mathbb{Z}[x]$. If $f$ is a permutation
	polynomial on $\mathbb{Z}_m$ and $g$ a null polynomial
	on $\mathbb{Z}_m$ then $f+g$ is a  permutation polynomial
	on $\Zmalfa$. In particular, for every null polynomial $g$
	on $\mathbb{Z}_m$, $x+g$ induces an element of $\Stab[{m}]$.
\end{corollary}

\section{The stabilizer of $\mathbb{Z}_{p^n}$\\ in
	the group of polynomial permutations of $\Zmalfa[p^n]$}\label{sec5}

Recall from Definition~\ref{std} that $\Stab[m]$ denotes the
pointwise stabilizer of $\mathbb{Z}_{m}$ in the group of polynomial
permutations on $\Zmalfa$.
We have seen in Theorem~\ref{14} the importance of this subgroup
for counting polynomial functions and polynomial permutations
on $\Zmalfa$. The somewhat technical results on $\Stab[m]$ that
we develop in this section will allow us to determine its order
and, from that, to derive explicit formulas for the number
of polynomial functions polynomial and permutations on
$\mathbb{Z}_{p^n}[\alpha]$ for $n\le p$ in section~\ref{sec6}.

We have already defined the ideal of null-polynomials and the
ideal of polynomials that are null together with their first
derivative in Section~\ref{sec2} (Definition~\ref{020}). For
counting purposes, we now pay special attention to the degrees
of the polynomials inducing the null function. We are interested
in the case of $R=\mathbb{Z}_{p^n}$ for $n>1$ (finite fields
having been covered already).
\begin{definition} \label{11.12}
	Let
	\begin{align*}
		\N{m} &=\{f\in \mathbb{Z}_m[x]\mid f\in \Null[\mathbb{Z}_m]
		\textnormal{ and }\deg f < k\},
		\\
		\Nd{m} &=\{f\in \mathbb{Z}_m[x]\mid
		f\in \Nulld[\mathbb{Z}_m] \textnormal{ and }\deg f < k\}.
	\end{align*}
\end{definition}

Recall from Definition~\ref{Algfun} that
$[f]_m$, short for  $[f]_{\mathbb{Z}_m}$, denotes the polynomial function
induced by $f$ on $\mathbb{Z}_m$.
\begin{proposition}\label{12}
	Let  $m=p_1^{n_1}\cdots p_l^{n_l}$, where $p_1,\ldots,p_l$
	are distinct primes and suppose that $n_j>1$ for $j=1,\ldots,l$.
	Then
	\begin{enumerate}
		\item[\rm (1)]%%\label{scndstab}
		$ |\Stab[m]|  =|\{[f']_m\mid f\in \Null[\mathbb{Z}_m] \}|$
		\item[\rm (2)] %%\label{nullpolydegk}
		if there exists a monic polynomial in $\mathbb{Z}[x]$ of
		degree~$k$ that is a null polynomial on $\Zmalfa$, then
		\begin{enumerate}
			\item[\rm (a)]%%\label{thirdstab}
			$
			|\Stab[m]|  =
			|\{[f']_m\mid f\in\Null[\mathbb{Z}_m]\textnormal{ with }\deg f<k\}|
			$
			\vskip5pt
			\item[\rm (b)]%%\label{fourthstab}
			$
			|\Stab[m]|  = [\Null[\mathbb{Z}_m]:\Nulld[\mathbb{Z}_m]]=
			\frac{|\N[k]{m}|}{|\Nd[k]{m}|}.
			$
		\end{enumerate}
	\end{enumerate}
\end{proposition}
\begin{proof}
	Ad (1).
	We define a bijection $\varphi$ from
	$\Stab[m]$  to the set of functions induced on
	$\mathbb{Z}_m$ by the derivatives of null polynomials on
	$\mathbb{Z}_m$.
	Given $F\in\Stab[m]$, let $h\in \mathbb{Z}[x]$ be (such as we know
	to exist by Proposition~\ref{firststab})
	a null polynomial on $\mathbb{Z}_m$ such that $x+h(x)$ induces $F$.
	We set $\varphi(F)=[h']_m$. Now Corollary~\ref{6}
	shows	$\varphi$ to be well-defined and injective, and
	Corollary~\ref{11.2} shows it to be surjective.
	
	Ad (2a). If $g\in \Null[\mathbb{Z}_m]$, then by
	Proposition~\ref{sur}, there exists $f\in \mathbb{Z}_m[x]$ with
	$\deg f<k$ such that $[f]_m=[g]_m$ (that is, $f\in \Null[\mathbb{Z}_m]$)
	and $[f']_m=[g']_m$.
	
	Ad (2b).
	Define
	$\varphi\colon \Null[\mathbb{Z}_m] \rightarrow \FZm$
	by $\varphi(f)=[f']_m$.
	Clearly, $\varphi$ is a  homomorphism of additive groups.
	Furthermore,
	$ \ker\varphi=\Nulld[\mathbb{Z}_m]$  and
	$\im\varphi=\{[f']_m\mid f\in \Null[\mathbb{Z}_m]\}$.
	By (1),
	\[
	|\Stab[m]|=[\Null[\mathbb{Z}_m]:\Nulld[\mathbb{Z}_m]].
	\]
	For evaluating the ratio, we restrict $\varphi$ to the additive
	subgroup of $\mathbb{Z}_m[x]$ consisting of polynomials
	of degree less than $k$ and get a homomorphism of additive
	groups defined on $\N[k]{m}$, whose image is still
	$\{[f']_m\mid f\in \Null[\mathbb{Z}_m]\}$, by Corollary~\ref{6},
	and whose kernel is $\Nd[k]{m}$. Hence
	\[
	|\Stab[m]|=[\N[k]{m}:\Nd[k]{m}].
	\]
\end{proof}

We now substitute concrete numbers from Theorem~\ref{05} and
Proposition~\ref{5} for the $k$ that stands for the degree of a
monic null polynomial on $\Zmalfa$ in
Proposition~\ref{12}~(2). Here, as in
Definition~\ref{mudef}, $\mu(m)$ denotes the smallest
positive integer whose factorial is divisible by $m$.
\begin{corollary}
	Let  $m=p_1^{n_1}\cdots p_k^{n_k}$, where $p_1,\ldots,p_k$ are
	distinct primes and suppose that $n_j>1$ for $j=1,\ldots,k$. Then
	\begin{enumerate}
		\item[\rm (1)]
		$|\Stab[m]|=|\{[f']_m\mid f\in \Null[\mathbb{Z}_m]
		\text{ with }\deg f<2\mu(m)\}|$
		\vskip5pt
		\item[\rm (2)]
		$\displaystyle|\Stab[m]|=
		\frac{|\N[2\mu(m)]{m}|}{|\Nd[2\mu(m)]{m}|}$.
	\end{enumerate}
\end{corollary}
\begin{corollary}\label{13}
	For a prime number $p$ and a natural number $n$, where $1<n\le p$,
	we have
	\begin{enumerate}
		\item[\rm (1)]
		$|\Stab[p^n]|=|\{[f']_{p^n}\mid f\in \Null[\mathbb{Z}_{p^n}]
		\text{ with } \deg f<(n+1)p\}|$
		\vskip5pt
		\item[\rm (2)]
		$\displaystyle|\Stab[p^n]|=
		\frac{|\N[(n+1)p]{p^n}|}{|\Nd[(n+1)p]{p^n}|}$.
	\end{enumerate}
\end{corollary}
\begin{remark}
	When $m=p$ is a prime, Proposition~\ref{12} and its Corollaries
	do not apply. This case has been treated in Theorem~\ref{1301}.
\end{remark}

We now employ Proposition~\ref{12} to show that Corollary~\ref{mshi}
takes a simpler form for polynomial functions on $\mathbb{Z}_{p^n}$,
when $n>1$. (Again, the case $n=1$ is exceptional, see Theorem~\ref{1301}.)
\begin{corollary}\label{14c}
	Let $n>1$. Then for any fixed $F\in \PolFun [\mathbb{Z}_{p^n}]$,
	\[
	|\Stab[p^n]|=|\{([f]_{p^n},[f']_{p^n})\mid
	f\in\mathbb{Z}[x] \text{ with\/ } [f]_{p^n}=F \}|.
	\]
\end{corollary}
\begin{proof}
	Set
	\[A=\{([f]_{p^n},[f']_{p^n})\mid
	f\in\mathbb{Z}[x] \text{ with\/ } [f]_{p^n}=F \},\] and
	fix $f_0\in \mathbb{Z}[x]$ with $[f_0]_{p^n}=F $. Then,
	$f-f_0$ is a null polynomial on $\mathbb{Z}_{p^n}$
	for any $f\in\mathbb{Z}[x]$ with $([f]_{p^n},[f']_{p^n})\in A$.
	
	We define a bijection
	\[
	\phi\colon A \rightarrow \{[h']_{p^n}\mid
	h\in \Null[\mathbb{Z}_{p^n}]\},\qquad
	\phi(([f]_{p^n},[f']_{p^n}))=[(f-f_0)']_{p^n} .
	\]
	Since $[(f-f_0)']_{p^n} = [f']_{p^n} - [f_0']_{p^n}$,
	$\phi$ is well defined.
	Also, $\phi$ is injective, because, for two different elements of $A$,
	$([f_1]_{p^n},[f_1']_{p^n})\ne ([f]_{p^n},[f']_{p^n})$ implies
	$[f_1']_{p^n}\ne [f']_{p^n}$ and hence
	$[(f_1-f_0)']_{p^n}\ne [(f-f_0)']_{p^n}$.
	
	To see that $\phi$ is surjective, consider
	$[h']_{p^n}$, where $h\in \Null[\mathbb{Z}_{p^n}]$.
	Then  $[f_0+h]_{p^n}=F$ and, therefore,
	$([f_0+h]_{p^n},[f'_0+h']_{p^n})$ is in $A$ and maps to
	$[h']_{p^n}$ under $\phi$.
	
	By Proposition~\ref{12}~(1),
	\[
	|\Stab[{p^n}]| = |\{[f']_{p^n}\mid f\in \Null[\mathbb{Z}_{p^n}]\}| = |A| .
	\]
\end{proof}
\begin{remark}
	Let $n=1$ and
	$A=\{([f]_{p^n},[f']_{p^n})\mid
	f\in\mathbb{Z}[x] \text{ with\/ } [f]_{p^n}=F \}$.
	Then $|A|=p^p$ by Lemma~\ref{Perf}, but
	$|\Stab[{p^n}]|=(p-1)^p$ by Theorem~\ref{1301}.
	This shows that the condition on $n$ in Corollary~\ref{14c} is necessary.
\end{remark}

We now we give a self-contained proof of Corollary~\ref{14a}
(not using Jiang's ratio \cite{ratio}, but emulating the argument
in the proof of Theorem~\ref{14}), for $R=\Zpnalfa$.

\begin{corollary}\label{14b}
	For any integer $n> 1$,
	\[
	|\PolFun[\Zpnalfa]|=|\PolFun[{\Zm[p^n]}]|^2\cdot |\Stab[p^n]|.
	\]
\end{corollary}
\begin{proof}
	Set
	\[
	B\; =
	\!\!\!\bigcup\limits_{ F\in \PolFun[{\Zm[p^n]}] }
	\!\!\!\{([f]_{p^n},[f']_{p^n})\mid [f]_{p^n}=F
	\text{ and } f\in\mathbb{Z}[x] \}.
	\]
	By Corollary~\ref{14c},
	\[|B|=|\PolFun[{\Zm[p^n]}]|\cdot |\Stab[p^n]|.\]
	
	We now define a function
	$\psi\colon \PolFun[\Ralfa] \rightarrow
	B\times \PolFun[{R}]$
	as follows: if $G\in\PolFun[\Ralfa]$ is induced by
	$g=g_1+\alfa g_2$, where $g_1,g_2 \in \Zm[p^n][x]$,
	we let $\psi(G)=(([g_1]_{p^n},[g'_1]_{p^n}),[g_2]_{p^n})$.
	
	By Corollary~\ref{6}, $\psi$ is well-defined and bijective, and, hence,
	$|\PolFun[\Zpnalfa]|=|B|\cdot|\PolFun[{\Zm[p^n]}]|$ .
\end{proof}

As $|\PolFun[{\Zm[p^n]}]|$ is a well-known quantity (quoted in
the introduction in Equation~(\ref{Kemp})), all we now need for
an explicit formula for $\mathcal{F}(\mathbb{Z}_{p^n}[\alpha])$
%%$|\PolFun[{\Zm[p^n]}{[\alpha]})|$
is an expression
for $|\Stab[p^n]|$. We will derive one for $n\le p$ in the next section.

\section{On the number of polynomial functions on $\Zpnalfa$}\label{sec6}

In this section we find explicit counting formulas for
the number of polynomial functions and the number of polynomial
permutations on %%$\Zmalfa[p^n]$
$\mathbb{Z}_{p^n}[\alpha]$
for $n\le p$. The reason for the
assumption $n\le p$ is that in this case (unlike the case $n>p$)
the ideal of null polynomials on $\mathbb{Z}_{p^n}$ is equal to
$((x^p-x),p)^n$. The equality can be seen by a counting argument
\cite[Corollary~2.5]{per2} ---
the ideal $((x^p-x),p)^n$ is clearly contained in $N_{\mathbb{Z}_{p^n}}$,
and, for $n\le p$, their respective indices in $\mathbb{Z}_{p^n}[x]$
are the same --- but it can also be derived from other
results~\cite[Theorem~3.3 (2)]{Zh04pfpp}.

This fact allows us to see at a glance if a polynomial is a null polynomial
modulo $p^k$ (for any $k\le n$) once we have expanded the polynomial
as a $\mathbb{Z}[x]$-linear combination of the powers $(x^p-x)^m$,
with coefficients of degree less than $p$. Our Lemma to this
effect, Lemma~\ref{15}, is taken from an earlier paper~\cite{per2}.
\begin{remark}\label{usefulrep}
	Let $R$ be a commutative ring and $h\in R[x]$ monic with $\deg h = q >0$.
	\begin{enumerate}
		\item%%\label{expansion}
		Every polynomial $f\in R[x]$ can be represented uniquely as
		\[f(x)=f_0(x)+f_1(x)h(x)+f_2(x)h(x)^2+\ldots\]
		with $f_k\in R[x]$ and $\deg f_k<q$ for all $k\geq 0$.
		\item%%\label{reducedexpansion}
		Let $I$ an ideal of $R$.
		Let $f,g\in R[x]$, $f=\sum_{i} a_i x^i$
		and $g=\sum_{i} b_i x^i$ be expanded as in (1)
		with $f_k=\sum_{j=0}^{q-1} a_{jk} x^j$ and
		$g_k=\sum_{j=0}^{q-1} b_{jk} x^j$.
		Then
		\[
		%%\forall i\;\;
		a_i \equiv b_i \bmod I \quad\text{for all}\ i\quad\Longleftrightarrow\quad
		%%\forall j,k\;\;
		a_{jk} \equiv b_{jk} \bmod I\quad\text{for all}\ j,k.
		\]
	\end{enumerate}
	
	(1) follows easily from repeated division with
	remainder by $h(x)$ and the fact that quotient and remainder are unique
	in polynomial division. (2) follows from the
	uniqueness of the expansion applied to polynomials in $(R/I)[x]$.
\end{remark}
\begin{lemma}[{\cite[Lemma 2.5]{per2}}]\label{15}
	Let $p$ be a prime and $f\in \mathbb{Z}[x]$ represented as
	in Remark~\ref{usefulrep} with respect to $h(x)=x^p-x$.
	\[f(x)=f_0(x)+f_1(x)(x^p-x)+f_2(x)(x^p-x)^2+\ldots\]
	with $f_k\in \mathbb{Z}[x]$ and $\deg f_k<p$ for all $k\geq 0$.
	
	Let $n\le p$. Then $f$ is a null polynomial on
	$\mathbb{Z}_{p^n}$ if and only if $f_k\in p^{n-k}\mathbb{Z}[x]$
	for $0\le k\le n$.
	%%\qed
\end{lemma}
\begin{corollary}\label{cut}
	Let $n\le p$. Then $|\N[(n+1)p]{p^n}|=p^{\frac{n(n+1)p}{2}}$.
\end{corollary}
\begin{proof}
	We express $f\in\mathbb{Z}[x]$ with $\deg f <(n+1)p$
	as in Remark~\ref{usefulrep},
	Lemma~\ref{15},
	$f(x)=\sum_{k=0}^{n}f_k(x) (x^p-x)^k$,
	where $f_k(x)=\sum_{j=0}^{p-1} a_{jk} x^j$.
	
	By Lemma~\ref{15} and Remark~\ref{usefulrep}~(2),
	$|\N[(n+1)p]{p^n}|$ is equal to the number of ways to choose the
	$a_{jk}$ from a fixed system of representatives modulo $ p^n $,
	such that $ a_{jk}\equiv 0 \mod{p^{(n-k)}}$ for $k\le n$.
	This number is
	$\prod_{k=0}^{n}p^{kp}=p^{p\sum_{k=0}^{n}k}=p^{\frac{n(n+1)p}{2}}$.
\end{proof}
\begin{lemma}\label{-16}
	Let $f\in\mathbb{Z}[x]$, where
	$f(x)=\sum_{k\geq 0}f_k(x) (x^p-x)^k$
	such that $f_k(x)=\sum_{j=0}^{p-1} a_{jk} x^j$.
	If we expand $f'$ in a similar way,
	$f'(x)=\sum_{k\geq 0} \hat f_k(x) (x^p-x)^k$, where
	$\hat f_k(x)=\sum_{j=0}^{p-1} \hat a_{jk} x^j,$
	then the following relations hold for all $k\ge 0$
	\begin{equation}\label{eq}
		\begin{split}
			\hat a_{0k} &= (kp+1)a_{1k} - (k+1)a_{0\,k+1}
			\\
			\hat a_{jk} &= (kp+j+1)a_{j+1\,k} + (k+1)(p-1)a_{j\,k+1}\qquad
			\textnormal{ for } 1\le j\le p-2
			\\
			\hat a_{p-1\,k} &= (k+1)(p-1)a_{p-1\,k+1} +
			(k+1)pa_{0\,k+1}.
		\end{split}
	\end{equation}
\end{lemma}
\begin{proof} Consider
	\begin{equation}\label{f'expres}
		\left(f_k(x)(x^p-x)^k\right)'=f_k'(x)(x^p-x)^k -
		k f_k(x)(x^p-x)^{k-1} + kpx^{p-1} f_k(x) (x^p-x)^{k-1}.
	\end{equation}
	
	We rewrite the last term of Equation~(\ref{f'expres}) by
	expanding $x^{p-1}f_k(x)$ as $\sum_{j=0}^{p-1} a_{jk}x^{p+j-1}$
	and substituting $x^{j+1}+ x^j (x^p-x)$ for $x^{p+j}$,
	to get integer linear-combinations of terms $x^j(x^p-x)^k$.
	\begin{align*}
		kp x^{p-1}  f_k(x) (x^p-x)^{k-1} & =
		\sum_{j=0}^{p-1}kpa_{jk} x^{p+j-1}(x^p-x)^{k-1}
		\\
		& =\bigg(\sum_{j=1}^{p-1} kpa_{jk}x^{p+j-1} +
		kpa_{0k} x^{p-1}\bigg) (x^p-x)^{k-1}
		\\
		&=\bigg(\sum_{j=1}^{p-1} kpa_{jk}(x^j + x^{j-1}(x^p-x))
		+ kpa_{0k} x^{p-1}\bigg) (x^p-x)^{k-1}
		\\
		& =\bigg(\sum_{j=1}^{p-2} kpa_{jk}x^j + (kpa_{p-1\, k}
		+ kpa_{0k}) x^{p-1}\bigg) (x^p-x)^{k-1}
		\\
		&\qquad + \bigg(\sum_{j=0}^{p-2} kpa_{j+1\,k}x^{j}\bigg) (x^p-x)^k
	\end{align*}
	and, therefore,
	\begin{align}
		\left(f_k(x)(x^p-x)^k\right)' & =
		\bigg( -ka_{0k} + \sum_{j=1}^{p-2} k(p-1) a_{jk}x^j +
		(k(p-1) a_{p-1\,k} + kpa_{0k})x^{p-1}\bigg)(x^p-x)^{k-1}\nonumber
		\\
		& \qquad  +  \bigg(\sum_{j=0}^{p-2} (kp+j+1) a_{j+1\,k}x^j
		\bigg) (x^p-x)^k.
	\end{align}
	Thus $f'(x)  = \sum_{k\geq 0}(f_k(x)(x^p-x)^k)' =
	\sum_{k= 0}\hat f_k(x)(x^p-x)^k$, where
	
	\begin{align*}
		\hat f_k(x) & =	(kp+1)a_{1k} - (k+1)a_{0\,k+1} +
		\sum_{j=1}^{p-2} \big((kp+j+1) a_{j+1\, k}
		+(k+1)(p-1) a_{j\, k+1}\big) x^j
		\\
		&\qquad +	((k+1)(p-1) a_{p-1\,k+1} + (k+1)p a_{0\,k+1}) x^{p-1}.
	\end{align*}
	
	Finally, expressing the $\hat a_{jk}$ in terms of the $a_{jk}$, we get
	\begin{alignat*}{2}
		%\begin{split}
		\hat a_{0k} &= (kp+1)a_{1k} - (k+1)a_{0\,k+1},\\
		\hat a_{jk} &= (kp+j+1)a_{j+1\,k} + (k+1)(p-1)a_{j\,k+1}\qquad
		&&\textrm{ for } 1\le j\le p-2,
		\\
		\hat a_{p-1\,k} &= (k+1)(p-1)a_{p-1\,k+1} +
		(k+1)pa_{0\,k+1}\qquad  &&\textrm{ for } k\geq 0.%%\hfill\qedhere
		%\end{split}
	\end{alignat*}
\end{proof}

Let $f\in \mathbb{Z}[x]$, $p$ a prime and $n\le p$. We are now in a
position to tell from the coefficients of the expansion of $f$ with
respect to powers of $(x^p-x)$ (as in Remark~\ref{usefulrep})
whether both $f$ and $f'$ are null polynomials on $\mathbb{Z}_{p^n}$.

\begin{theorem}\label{160}
	Let $n\le p$ and
	$f(x)=\sum_{k=0}^{m}f_k(x)(x^p-x)^k\in\mathbb{Z}[x]$,
	where $f_k(x)=\sum_{j=0}^{p-1} a_{jk} x^j$.
	
	Then $f$ and $f'$ are both null polynomials on $\mathbb{Z}_{p^n}$
	if and only if, for $1\leq k<\min(p,n+1)$,
	\begin{equation}\label{eiq1}
		\aligned
		a_{j0}& \equiv 0 \pmod{p^n}
		\\
		a_{jk}& \equiv  0 \pmod{p^{n-k+1}}.
		\endaligned
	\end{equation}
\end{theorem}
\begin{proof}
	($\Rightarrow$) Suppose $f$ and $f'$ are null polynomials
	on $\mathbb{Z}_{p^n}$. Then $f'(x)=\sum_{k= 0}^{m} \hat f_k(x) (x^p-x)^k$,
	with $\hat f_k(x)=\sum_{j=0}^{p-1} \hat a_{jk} x^j$, such that,
	by Lemma~\ref{-16}, the coefficients $ a_{jk} $ and $\hat a_{jk}$ satisfy
	Equation~(\ref{eq}).
	Since $f'$ is a null polynomial on $\mathbb{Z}_{p^n}$,
	Lemma~\ref{15} implies, for $j=0,\ldots,p-1$,
	\begin{equation}\label{0cof}
		\hat a_{jk}\equiv 0\pmod{p^{n-k}}\qquad\text{ for }k\le n .
	\end{equation}
	
	Again by Lemma~\ref{15}, it is clear that
	\begin{equation}\label{dreivcofn}
		a_{j0}\equiv 0\pmod{p^n}\qquad \text{ for }j=0,1,\ldots,p-1.
	\end{equation}
	
	For $1\le k<\min(p,n+1)$, we use induction.
	To see $a_{j1}\equiv 0\pmod{p^n}$,
	we set $k=0$ in  Equation~(\ref{eq}), and get
	\begin{equation}\label{k=0}
		\begin{split}
			\hat a_{00} &= a_{10} - a_{0\,1},
			\\
			\hat a_{j0} &= (j+1)a_{j+1\,0} + (p-1)a_{j\,1}
			\qquad\textrm{for}\  1\le j\le p-2,
			\\
			\hat a_{p-1\,0} &= (p-1)a_{p-1\,1} + pa_{0\,1}.
		\end{split}
	\end{equation}
	From Equations~(\ref{0cof}),~(\ref{dreivcofn}), and~(\ref{k=0}),
	we conclude that $a_{j1}\equiv 0\pmod{p^n},\ j=0,1,\ldots,p-1$.
	
	Now, for $2\le k+1<\min(p,n+1)$, we prove the statement for $k+1$
	under the hypothesis
	\begin{equation}\label{indhypo}
		a_{jk}\equiv 0 \pmod{p^{n+1-k}}\qquad \text{ for }j=0,1,\ldots,p-1.
	\end{equation}
	We rewrite Equation~(\ref{eq}) as
	\begin{alignat}{2}\label{eqmodify}
		%\begin{split}
		(k+1)a_{0\,k+1} &=(kp+1)a_{1k} -\hat a_{0k}
		\nonumber\\
		(k+1)(p-1)a_{j\,k+1} &=\hat a_{jk}  - (kp+j+1)a_{j+1\,k}
		\quad&&\textnormal{for }\  1\le j\le p-2
		\\
		(k+1)(p-1)a_{p-1\,k+1} &= \hat a_{p-1\,k} -
		(k+1)pa_{0\,k+1}\quad&&\textnormal{for }\	k=0,1,\ldots,n-1.\nonumber
		%\end{split}
	\end{alignat}
	Since $k+1<p$ and $n+1-k>n-k$, Equations~(\ref{eqmodify}),~(\ref{0cof})
	and the induction hypothesis (Equation~(\ref{indhypo})) give
	\[
	a_{j\,k+1}\equiv 0 \pmod{p^{n-k}}\qquad \text{ for } j=0,1,\ldots,p-1.
	\]
	
	For $k\geq\min(p,n+1)$, we note that
	$(x^p-x)^k\in \Nulld[\mathbb{Z}_{p^n}]$\!. Hence
	$f_k(x)(x^p-x)^k\in  \Nulld[{\mathbb{Z}_{p^n}}]$.
	So, there are no restrictions on $a_{jk}$ for $j=0,\ldots,p-1$.
	
	($\Leftarrow$) Assume that (\ref{eiq1}) is true. Then, for $k\le p$,
	$a_{jk}\equiv 0 \pmod{p^{(n-k)}}$ since $n+1-k>n-k$.
	We use Lemma~\ref{-16} and Equation (\ref{eiq1}) to show that
	$\hat a_{jk}\equiv 0\pmod{p^{(n-k)}}$ for $0\le k\le p$.
	The result now follows by Lemma~\ref{15}.
\end{proof}
\begin{corollary} \label{15.1}
	Let $n\le p$ and $r=\min(n+1, p)$, that is,
	$r=\begin{cases}
	n+1 & \textnormal{ if } n<p\\
	p & \textnormal{ if } n=p
	\end{cases}.$
	
	Then $(x^p-x)^r$ is a monic null polynomial on
	$\mathbb{Z}_{p^n}[\alfa]$ of minimal degree.
\end{corollary}
\begin{proof}
	By Lemma~\ref{31}, $(x^p-x)^r$ is a null polynomial on $\Zpnalfa$.
	Let $h\in \mathbb{Z}[\alfa][x]$ be a null polynomial on
	$\Zpnalfa$ with  $\deg h<rp$. By Corollary~\ref{nullrestat},
	it suffices to consider $h\in \mathbb{Z}[x]$. We show that $h$
	is not monic. If $h=0$ this is evident. If $h\ne 0$, we
	expand  $h$ as in Lemma~\ref{15}:
	\[
	h(x)=h_0(x)+h_1(x)(x^p-x)+\cdots+ h_{r-1}(x)(x^p-x)^{r}
	\]
	with $h_k(x)=\sum_{j=0}^{p-1} a_{jk} x^j\in \mathbb{Z}[x]$.
	By Theorem~\ref{160}, it follows  that for $0\le j\le p-1$
	\begin{align*}
		a_{j0}& \equiv 0 \pmod{p^{n}},
		\\
		a_{jk}& \equiv 0 \pmod{p^{(n-k+1)}}\quad
		\text{for }  1\leq k<r. \nonumber
	\end{align*}
	If $l$ is  the largest number such that $h_l(x)\ne 0$,
	then $a_{p-1\,l}\ne 1$, since
	$a_{p-1\,l}\equiv 0 \mod{p^{(n-l+1)}}$. Thus $h$ cannot be monic.
\end{proof}

Recall from Definitions~\ref{020} and \ref{11.12} that
$f\in \Nd[(n+1)p]{p^n}$ means $f$ and $f'$ are null
polynomials on $\mathbb{Z}_{p^n}$ and $\deg f<(n+1)p$.												 \begin{corollary}\label{16}
	Let $n\le p$. Then
	$|\Nd[(n+1)p]{p^n}|=\begin{cases}p^{\frac{n(n-1)p}{2}}
	& \textnormal{if } n<p
	\\[1ex]
	p^ {\frac{(n^2-n+2)p}{2}} & \textnormal{if } n=p
	\end{cases}.$
\end{corollary}
\begin{proof}
	We represent every polynomial $f\in \mathbb{Z}_{p^n}[x]$
	with $\deg f < (n+1)p$ uniquely, by Remark~\ref{usefulrep}, as
	\[
	f(x)=\sum_{k= 0}^{n}f_k(x) (x^p-x)^k \qquad\textrm{with}\
	f_k(x)=\sum_{j=0}^{p-1} a_{jk} x^j \in\mathbb{Z}_{p^n}[x].
	\]
	By Theorem~\ref{160}, counting the polynomials in $\Nd[(n+1)p]{p^n}$
	amounts to counting the number of choices for the $a_{jk}$ such that
	$a_{j0} \equiv  0 \mod{p^n}$ and $a_{jk} \equiv 0 \mod{p^{n-k+1}}$
	for $1\leq k<\min(p,n+1)$ and $0\le j\le p-1$.
	
	When $n<p$, there are $p^{k-1}$ choices for $a_{jk}$ for
	each pair $(j,k)$ with $1\le k\le n$ and $0\le j\le p-1$.
	Hence the total number of ways of choosing all coefficients,
	when $n<p$, is equal to
	\[
	\prod_{k=1}^{n}p^{p(k-1)}=\prod_{k=0}^{n-1}p^{pk}=
	p^{p\sum_{k=0}^{n-1}k}=p^{\frac{pn(n-1)}{2}}.
	\]
	When $n=p$, $a_{jn}$ can be chosen in $p^n$ ways, and the resulting
	total is
	\begin{equation*}
		p^{np}\prod_{k=1}^{n-1}p^{p(k-1)}=p^{np}\prod_{k=0}^{n-2}p^{pk}
		=p^{np+p\sum_{k=0}^{n-2}k}=p^{\frac{p(n^2-n+2)}{2}}.\qedhere
	\end{equation*}
\end{proof}

At last, we  obtain an explicit formula for the order of
$\Stab[p^n]$ for $n\le p$.
\begin{theorem} \label{st}
	Let $1\le n\le p$. Then
	\[
	|\Stab[p^n]|= \begin{cases}
	(p-1)^p &   \textnormal{if }  n=1
	\\
	p^{np} & \textnormal{if } 1<n<p
	\\
	p^{(n-1)p}  & \textnormal{if } n=p
	\end{cases}.
	\]
\end{theorem}
\begin{proof}
	The case $n=1$ is a special case of Theorem~\ref{1301}~(3).
	Let $1<n\le p$.	By Corollary~\ref{13},
	\begin{align*}
		|\Stab[p^n]| & =\frac{|\N[(n+1)p]{p^n}|}{|\Nd[(n+1)p]{p^n}|}.
	\end{align*}
	Now Corollaries~\ref{cut} and \ref{16}, respectively, say that
	\begin{equation*}
		|\N[(n+1)p]{p^n}|=p^\frac{n(n+1)p}{2} \qquad\textrm{and}\qquad
		|\Nd[(n+1)p]{p^n}|=
		\begin{cases}p^{\frac{n(n-1)p}{2}} & \textnormal{if } n<p
			\\
			p^ {\frac{(n^2-n+2)p}{2}} & \textnormal{if } n=p
		\end{cases}.\qedhere
	\end{equation*}
\end{proof}
\begin{example}
	Let $R=\mathbb{Z}_{4}$. Then $|\Stab[4]|=4$ by Theorem~\ref{st}.
	Now, by Corollary~\ref{15.1}, the polynomial $(x^2-x)^2$ is a
	monic 	null polynomial on $\mathbb{Z}_{4}[\alfa]$ of minimal degree.
	So every polynomial function on  $\mathbb{Z}_{4}[\alfa]$ can
	be represented by a polynomial of degree less than $4$.
	Consider the following null polynomials on $\mathbb{Z}_{4}$:
	\[
	f_1=0,\quad f_2 =2(x^2-x),\quad f_3=2(x^3-x),\quad f_4= 2(x^3-x^2).
	\]
	It is evident that $[x+f_i]_4= id_{\mathbb{Z}_{4}}$, and so by
	Corollary~\ref{corof11},
	$[x+f_i]\in \Stab[4]$, where $[x+f_i]$ denotes the function
	induced by $x+f_i$ on $\mathbb{Z}_{4}[\alfa]$ for $i=1,\ldots,4$.
	Note that  $[1+f'_i]_4\ne [1+f'_j]_4$, however,
	and hence by Corollary~\ref{Gencount},  $[x+f_i]\ne [x+f_j]$ whenever
	$i\ne j$. Therefore  $\Stab[4]=\{[x+f_i],\linebreak i=1,\ldots,4\}$.
	Actually,  $\Stab[4]$ is the Klein 4-group.
\end{example}

Theorem~\ref{st} now allows us to state explicit formulas for the
number of polynomial functions and polynomial permutations on
$\Zpnalfa$ for $n\le p$. Our formula for $|\PrPol[\Zpnalfa]|$
depends on $p$ and $n$. To understand it in terms of the residue
field and nilpotency of the maximal ideal of $\Zpnalfa$,
recall from Proposition~\ref{0} that the residue field of $\Zpnalfa$
is isomorphic to $\mathbb{Z}_{p}$ and the nilpotency of the
maximal ideal is $n+1$.
\begin{theorem} \label{countperm}
	Let  $1\le n\le p$. Then  the number $|\PrPol[\Zpnalfa]|$ of
	polynomial permutations on $\Zmalfa[p^n]$  is given by
	\begin{equation*}
		|\PrPol[\Zpnalfa]|=
		\begin{cases}
			p!(p-1)^pp^{(n^2+2n-2)p} & \textnormal{if } n<p
			\\[1ex]
			p!(p-1)^pp^ {(n^2+2n-3)p} & \textnormal{if } n=p
		\end{cases}.
	\end{equation*}
\end{theorem}
\begin{proof}
	The case $n=1$ is covered by Proposition~\ref{13.111}. Now,
	let $1<n\le p$. Using that $\mu(p^k)=kp$ for $k\le p$,
	we simplify the formulas for $|\PolFun[\mathbb{Z}_{p^n}]|$
	and $|\PrPol[\mathbb{Z}_{p^n}]|$ quoted in the introduction
	(Equation~(\ref{Kemp})) accordingly.  For $1<n\le p$,
	\begin{equation} \label{simpliefiedKemp}
		\left| \PolFun[\mathbb{Z}_{p^n}] \right| =
		p^{\frac{n(n+1)p}{2}}
		\qquad\hbox{ and }\qquad
		\left| \PrPol[\mathbb{Z}_{p^n}] \right| =
		p! (p-1)^p p^{-2p}
		p^{\frac{n(n+1)p}{2}}.
	\end{equation}
	
	Substituting the formula from Theorem~\ref{st} for $|\Stab[p^n]|$
	and the above expressions for $|\PolFun[\mathbb{Z}_{p^n}]|$
	and $|\PrPol[\mathbb{Z}_{p^n}]|$ in Theorem~\ref{14},
	we obtain  the desired result.
\end{proof}
\begin{theorem} \label{countfunc}
	Let $n\le p$.
	The number $|\PolFun[\Zpnalfa]|$ of polynomial functions
	on $\Zmalfa[p^n]$   is given by
	\[
	|\PolFun[\Zpnalfa]|=
	\begin{cases}p^{(n^2+2n)p} & \textnormal{if } n<p
	\\[1ex]
	p^ {(n^2+2n-1)p} & \textnormal{if } n=p
	\end{cases}.
	\]
\end{theorem}
\begin{proof}
	The case $n=1$ is covered by Corollary~\ref{fieldfunc}. For $1< n\le p$, we
	substitute the expression from Theorem~\ref{st} for $|\Stab[p^n]|$ and
	the formula for $|\PolFun[\mathbb{Z}_{p^n}]|$ from Equation~(\ref{Kemp})
	(simplified as in Equation~(\ref{simpliefiedKemp}) in the proof of
	Theorem~\ref{countperm}) in Corollary~\ref{14b}.
\end{proof}

\section{A canonical form}\label{sec7}

In this section we find a canonical representation for the polynomial
functions on $\Zmalfa[p^n]$ whenever $n\le p$. As before
(see Definition~\ref{mudef}), $\mu(m)$ stands for the
smallest natural number $n$ such that $m$ divides $n!$.

\begin{lemma}[{\cite[Theorem 10]{pol2}}]\label{08011}
	Let $F$ be a polynomial function on $\mathbb{Z}_m$.
	Then $F$ is uniquely represented by a polynomial $f\in\mathbb{Z}[x]$
	of the form
	\[
	f(x)= \sum_{i=0}^{\mu(m)-1}a_ix^i\qquad\text{ with }\
	0\le a_i< \frac{m}{\gcd(m,i!)}.
	\]
\end{lemma}
\begin{proposition}\label{8}
	Let $F\colon \Zmalfa\rightarrow \Zmalfa$ be a polynomial function
	on $\Zmalfa$. Then $F$ can be represented by a polynomial
	$f\in\mathbb{Z}[x]$  of the form
	\[
	f(x)= \sum_{i=0}^{2\mu(m)-1}a_ix^i+
	\sum_{j=0}^{\mu(m)-1}b_jx^j\alfa
	\qquad\text{with}\
	0\le a_i,b_j<m\ \
	\text{and}\ \
	0\le b_j< \frac{m}{\gcd(m,j!)}
	\]
	and the $b_j$ in such a representation are unique.
\end{proposition}
\begin{proof}
	By Corollary~\ref{7}, $F$ can be represented by a polynomial
	$g_1+\alfa g_2 $, where
	\[
	g_1(x)=\sum_{i=0}^{2\mu(m)-1}c_ix^i \qquad\textrm{and}\qquad
	g_2(x)=\sum_{j=0}^{\mu(m)-1}d_jx^j
	\]
	with $c_i,d_j \in \mathbb{Z}$.
	Choosing  $a_i, b_j$ to be the smallest
	non-negative integers such that $c_i\equiv a_i$ and
	$d_j\equiv b_j \mod{m}$, we see that $F$ is represented by
	\[
	g(x)=\sum_{i=0}^{2\mu(m)-1}a_ix^i+
	\sum_{j=0}^{\mu(m)-1}b_jx^j\alfa
	\]
	with $0\le a_i,b_j<m$.
	Now, since $\Zmalfa$ is a $\mathbb{Z}$-algebra, substituting elements
	of $\Zmalfa$ for the variable $x$ in $g$ defines a function on $\Zmalfa$.
	For $k,l\in \mathbb{Z}_m$, we have
	\[
	g(k+l\alfa)=\sum_{i=0}^{2\mu(m)-1}a_i(k+l\alfa)^i+
	\sum_{j=0}^{\mu(m)-1}b_jk^j\alfa.
	\]
	By Corollary~\ref{Gencount}, $F$ depends on the function induced by
	$\sum_{j=0}^{\mu(m)-1}b_jx^j$ on $ \mathbb{Z}_m$ but not on the
	function induced by its derivative. So we can replace
	$\sum_{j=0}^{\mu(m)-1}b_jx^j$ by any polynomial $h\in\mathbb{Z}[x]$
	such that $[\sum_{j=0}^{\mu(m)-1}b_jx^j]_m=[h]_m$.
	Hence, by Corollary~\ref{Gencount} and
	Lemma~\ref{08011}, $b_j$ can be chosen uniquely such that
	$0\le b_j< \frac{m}{\gcd(m,j!)}$.
\end{proof}

By combining Proposition~\ref{8} with Proposition~\ref{5}, we
obtain the following corollary.
\begin{corollary}\label{08001}
	Let $p$ be a prime number and $n\le p$ a positive integer.
	Let $F\colon \Zmalfa[p^n]\rightarrow \Zmalfa[p^n]$
	be a polynomial function on $\Zmalfa[p^n]$. Then $F$
	can be represented as a polynomial
	$f(x)= \sum_{i=0}^{(n+1)p-1}a_ix^i+
	\sum_{j=0}^{np-1}b_jx^j\alfa$
	with $0\le a_i,b_j< p^n$.
	Moreover, $b_j$ can be chosen uniquely
	such that $0\le b_j< \frac{p^n}{\gcd(p^n,j!)}$.
\end{corollary}

Finally, we give a canonical representation
for polynomial functions on $\Zmalfa[p^n]$ for $n\le p$.
\begin{theorem}\label{canon}
	Let $n\le p$. Every polynomial function $F$ on $\Zmalfa[p^n]$ is
	uniquely represented by a polynomial  $f\in\mathbb{Z}[x]$ of the form
	\[f(x)=\sum_{k=0}^{m}f_k(x)(x^p-x)^k+
	\sum_{i=0}^{np-1}b_ix^i\alfa \quad\text{with}\quad
	f_k(x)=\sum_{j=0}^{p-1} a_{jk} x^j,\]
	where
	\begin{enumerate}
		\item[\rm (1)]
		$m=\min(n, p-1)$
		\item[\rm (2)]
		$ 0\le  a_{j0}  < p^n$ and $0\le  a_{jk}  < p^{n-k+1}$ \ \
		(for $j=0,\ldots, p-1$ and $k=1,\ldots, m$)
		\item[\rm (3)]
		$0\le b_{i} < \frac{p^n}{\gcd(p^n,i!)}$ \ \ (for $i=0,\ldots, np-1$).
	\end{enumerate}
\end{theorem}
\begin{proof}
	Let $F$ be a polynomial function on  $\Zmalfa[p^n]$.
	By Corollary~\ref{08001}, we can represent $F$ by $f=g+\alfa h$
	with $g,h \in \mathbb{Z}[x]$, such that $\deg g<(n+1)p-1$ and
	$h(x)=\sum_{i=0}^{np-1}b_ix^i$ with
	$0\le b_{i}<  \frac{p^n}{\gcd(p^n,i!)}$;
	and the coefficients $b_i$ in such a representation are unique.
	
	By Corollary~\ref{15.1}, $(x^p-x)^{m+1}$ is  null on $\Zmalfa[p^n]$.
	Thus  we can choose $g$ with $\deg g<p(m+1)$ by  Proposition~\ref{sur}.
	We expand $g$ as in Lemma~\ref{15},
	$g(x)=\sum_{k=0}^{m}g_k(x)(x^p-x)^k$, where
	$g_k(x)=\sum_{j=0}^{p-1} c_{jk} x^j\in \mathbb{Z}[x]$.
	
	By division with remainder, we get
	$c_{j0}=p^{n}q_{j0}+a_{j0}$ and $c_{jk}=p^{n-k+1}q_{jk}+a_{jk}$
	with $0\le a_{j0}<p^n$ and $0\le a_{jk}<p^{n-k+1}$
	for $j=0,\ldots,p-1$, and  $k=1,\ldots,m$.
	By  Theorem~\ref{160},
	\[
	p^n(x^p-x)\quv p^{n-k+1}(x^p-x)^k\quv 0 \qquad\text{on}\ \ \Zmalfa[p^n].
	\]
	Thus, if we set  $f_k(x)=\sum_{j=0}^{p-1} a_{jk} x^j$
	for  $k=0,\ldots, m$,
	we have, by  Corollary  \ref{6},
	\[
	g(x)= 	\sum_{k=0}^{m}g_k(x)(x^p-x)^k \quv
	\sum_{k=0}^{m}f_k(x)(x^p-x)^k  \qquad\text{ on }\ \ \Zmalfa[p^n],
	\]
	and  hence we can replace $g$ by
	$\sum_{k=0}^{m}\sum_{j=0}^{p-1} f_k(x)(x^p-x)^k$
	in the representation of the function $F$. Therefore $F$ is
	induced by $f=g+\alfa h$, where
	$g(x)=\sum_{k=0}^{m}\sum_{j=0}^{p-1} a_{jk} x^j(x^p-x)^k$\!,
	with  $0\le a_{j0}<p^n$, $0\le a_{jk}<p^{n-k+1}$ for
	$j=0,\ldots,p-1$, and	$k=1,\ldots,m$; and $h$ as above.
	To count the number of ways  of selecting such a polynomial $f$,
	we need to count the number of ways of choosing $g$ and $h$.
	First, we do that for $g$. We note that $f_0(x)$ can
	be determined in $p^{np}$ ways, since
	$a_{j0}<p^n$ for $j=0,\ldots,p-1$. While, if $1\le k\le m$,
	$f_k(x)$ can be selected in $p^{p(n-k+1)}$ ways,
	since $0\le a_{jk}<p^{n-k+1}$ for $j=0,\ldots,p-1$.
	So, the number of ways to choose $g$ is
	\[p^{np}\prod_{k=1}^{m}p^{p(n-k+1)}=
	p^{np}\prod_{k=0}^{m-1}p^{p(n-k)}.\]
	On the other hand, simple calculations show that
	$\sum_{i=0}^{np-1}b_ix^i\alfa$ can be chosen in
	$p^{\frac{pn(n+1)}{2}}$ ways, since
	$0\le b_{i}<  \frac{p^n}{\gcd(p^n,i!)}$.
	Thus the number  of ways that 	$f$ can be chosen is
	\[
	p^{np}\prod_{k=0}^{m-1}p^{p(n-k)}\cdot p^{\frac{pn(n+1)}{2}}=
	\begin{cases}p^{(n^2+2n)p} & \textnormal{if } n<p
	\\[1ex]
	p^{(n^2+2n-1)p} & \textnormal{if } n=p
	\end{cases}.
	\]
	By Theorem~\ref{countfunc}, this last quantity equals
	$|\PolFun[\Zpnalfa]|$ and, therefore, the representation
	is unique.
	\end{proof}
	
	\section*{Acknowledgement}
	The authors would like to thank Irena Swanson for valuable
	suggestions and comments on earlier versions of the manuscript.

\end{document}